\documentclass[reqno,12pt]{amsart}

\usepackage{amsmath,amsfonts,amsthm,amscd,upref,amstext}
\usepackage{amssymb}
\usepackage{mathtools}
\usepackage{xcolor}
\usepackage{fullpage}
\usepackage{lmodern}

\theoremstyle{plain}
\newtheorem{prop}{Proposition}[section]
\newtheorem{thm}[prop]{Theorem}
\newtheorem{cor}[prop]{Corollary}
\newtheorem{lem}[prop]{Lemma}

\theoremstyle{definition}
\newtheorem{defn}[prop]{Definition}
\newtheorem{ex}[prop]{Example}

\theoremstyle{remark}
\newtheorem{rem}[prop]{Remark}

\newcommand{\N}{\mathbb N}
\newcommand{\Z}{\mathbb Z}

\DeclareMathOperator{\Min}{Min}

\numberwithin{equation}{section}

\begin{document}

\title{Mixed Multiplicities and Rees Theorems for Graded Families and Their Converses}
\date{}

\author{T. H. Freitas}
\address{Universidade Tecnol\'ogica Federal do Paran\'a, 85053-525, Guarapuava-PR, Brazil}
\email{freitas.thf@gmail.com}

\author{V. H. Jorge P\'erez}
\address{Universidade de S{\~a}o Paulo -- ICMC, 13560-970, S{\~a}o Carlos-SP, Brazil}
\email{vhjperez@icmc.usp.br}


\begin{abstract}

Let $(R,\mathfrak m)$ be a Noetherian local ring and let
$\mathcal I^{(1)},\ldots,\mathcal I^{(s)}$ be graded
$\mathfrak m$-primary families. We introduce weak asymptotic and strong
homogeneous joint reductions and prove corresponding extensions of
Rees's mixed multiplicity theorem relating joint reductions and mixed
multiplicities. Actually, for weak asymptotic joint reductions, mixed
multiplicities are realized as normalized limits of Hilbert--Samuel
multiplicities, yielding an asymptotic Rees theorem. For multigraded
admissible collections, a strong homogeneous joint reduction
$\{x_{ij}\}$ of type $\mathbf d=(d_1,\ldots,d_s)$, with
$x_{ij}\in I_{a_{ij}}^{(i)}$, satisfies the formula
\[
e((x_{ij});R)
=
\left(\prod_{i,j}a_{ij}\right)
e\!\left(
\mathcal I^{(1)[d_1]},\ldots,
\mathcal I^{(s)[d_s]}
\right).
\]
 We also prove a localized converse under formal equidimensionality,
recovering the classical converse of Rees in the adic case. As an application, we derive 
a B\"oger-type theorem for graded families, whose adic specialization recovers
B\"oger's classical theorem. In contrast, the converse to the weak asymptotic Rees theorem
 fails even for an adic family on a
one-dimensional regular local ring.
\end{abstract}

\maketitle

\section{Introduction}

Mixed multiplicities provide a natural multigraded extension of
Hilbert--Samuel multiplicity and have become an important tool in
commutative algebra and algebraic geometry. Their origins can be traced
to the work of Bhattacharya \cite{Bhattacharya1957}, and the theory was
subsequently developed through the work of Teissier, Rees, Rees--Sharp,
and many others. An interesting fact is that mixed multiplicities have appeared in connection with
equisingularity, integral dependence, Milnor numbers, blow-up algebras,
and mixed volumes of convex bodies. We refer to
\cite{TrungVermaSurvey,SwansonHuneke} for broad accounts of the
classical theory and its applications.

Several authors have developed different aspects of the theory, including
inequalities \cite{ReesSharp1978}, multigraded and extended Rees
algebras \cite{KatzVerma1989}, positivity \cite{Trung2001}, arbitrary
ideals \cite{Viet2000}, and connections with mixed volumes of polytopes
\cite{TrungVerma2007}. These works show that mixed multiplicities encode
both algebraic and geometric information about collections of ideals. A particularly important connection with reduction theory was
established by Rees. In \cite{Rees1984}, he introduced joint reductions
and showed that mixed multiplicities of $\mathfrak m$-primary ideals can
be realized as Hilbert--Samuel multiplicities of ideals generated by
suitable joint reductions. In the notation of
\cite[Chapter~17]{SwansonHuneke}, if $d_1+\cdots+d_s=d$ and a collection
of $d_i$ elements from $J_i$ forms an appropriate joint reduction, then
its generated ideal $Q$ satisfies
\[
e(Q;R)=e(J_1^{[d_1]},\ldots,J_s^{[d_s]};R)
\]
(\cite[Theorem~17.4.9]{SwansonHuneke}). Its converse is more subtle: under
formal equidimensionality, localized mixed-multiplicity equalities,
together with suitable radical and height conditions, characterize
joint reductions (\cite[Theorem~17.6.1]{SwansonHuneke}).

More recently, the theory has been extended to filtrations and graded
families, including non-Noetherian settings. Actually, Cutkosky developed
asymptotic multiplicity theory for graded families
\cite{Cutkosky2013,Cutkosky2014}, while mixed multiplicities for
filtrations and related positivity and divisorial questions were studied
in
\cite{CutkoskySarkarSrinivasan,CutkoskyDivisorial,
CutkoskySrinivasanVerma,CutkoskySarkar2022}.
Also, Cid-Ruiz and Monta\~no subsequently developed mixed multiplicities for
graded families under mild hypotheses and proved a general
volume--multiplicity principle \cite{CidRuizMontano}. In 2026, Cutkosky \cite{Cutkosky2026} introduced a multiplicity for an arbitrary
graded $\mathfrak m$-primary family
$\mathcal I=\{I_n\}_{n\geq0}$ over an arbitrary Noetherian local ring by
showing that, 
\[
e(\mathcal I)
=
\lim_{n\to\infty}\frac{e(I_n)}{n^d}
\]
always exists (\cite[Theorem~1.3]{Cutkosky2026}). More importantly for the present paper, Cutkosky defined mixed
multiplicities of graded $\mathfrak m$-primary families as the
coefficients of a homogeneous polynomial governing their asymptotic
multiplicities (\cite[Theorem~1.6]{Cutkosky2026}). In particular,
\[
e\!\left(
\mathcal I^{(1)[d_1]},\ldots,
\mathcal I^{(s)[d_s]}
\right)
=
\lim_{n\to\infty}
\frac{
e(I_n^{(1)[d_1]},\ldots,I_n^{(s)[d_s]})
}{n^d}
\]
(\cite[Equation~(9)]{Cutkosky2026}). This formula will be used
throughout the paper. This multiplicity theory naturally raises a 
question. Rees's mixed-multiplicity theorem relates mixed
multiplicities of ideals to joint reductions. What should replace joint
reductions when powers of fixed ideals are replaced by arbitrary graded
families? Here, ``Rees theorem'' refers to this joint-reduction
interpretation of mixed multiplicities \cite{Rees1984}, not to the
classical criterion relating equality of multiplicities to integral
dependence.

In order to answer this question, we introduce two notions of joint reduction for
graded families. A \emph{weak asymptotic joint reduction} is defined
levelwise, by choosing classical joint reductions at sufficiently large
levels. A \emph{strong homogeneous joint reduction} is defined in the
multigraded Rees algebra by a fixed homogeneous system controlling all
sufficiently large multidegrees. For weak asymptotic joint reductions, Theorem~\ref{thm:weak-rees}
gives an asymptotic form of Rees's mixed-multiplicity theorem. More
precisely, if $\mathbf d=(d_1,\ldots,d_s)$ satisfies $|\mathbf d|=d$
and $Q_n$ is generated by a weak asymptotic joint reduction of type
$\mathbf d$ at level $n$, then
\[
e\!\left(
\mathcal I^{(1)[d_1]},\ldots,
\mathcal I^{(s)[d_s]}
\right)
=
\lim_{n\to\infty}\frac{e(Q_n;R)}{n^d}.
\]
For adic families this recovers the classical Rees theorem, while for
Noetherian graded families the result reduces to the classical case
after passing to a suitable Veronese. By contrast, the strong theory yields an exact formula.
Theorem~\ref{thm:weighted-strong-rees} proves that, for a multigraded
admissible collection and a strong homogeneous joint reduction
$\{x_{ij}\}$ of type $\mathbf d$, with
$x_{ij}\in I_{a_{ij}}^{(i)}$ and $Q=(x_{ij})$, one has
\[
e(Q;R)
=
\left(
\prod_{i=1}^{s}\prod_{j=1}^{d_i}a_{ij}
\right)
e\!\left(
\mathcal I^{(1)[d_1]},\ldots,
\mathcal I^{(s)[d_s]}
\right).
\]
When the residue field is infinite, admissibility guarantees such
reductions in degree one and, by power stability, in arbitrary
prescribed positive degrees. In the degree-one adic case, this formula
recovers Rees's classical mixed-multiplicity theorem
\cite[Theorem~17.4.9]{SwansonHuneke}. It is important to realize that, Proposition~\ref{prop:weak-not-strong} shows that the weak and strong
notions may differ. For Noetherian graded families, however, both
theories reduce through suitable Veronese subfamilies to the classical
adic setting (see Corollaries~\ref{cor:weak-rees-noetherian} and
\ref{cor:weighted-strong-rees-noetherian}).

To prove a converse to Rees's theorem, the two notions behave quite
differently. Example~\ref{ex:weak-converse} shows that the weak
asymptotic Rees formula has no converse in general, even for an adic
family on a one-dimensional regular local ring. For strong homogeneous joint reductions,
Theorem~\ref{thm:localized-strong-converse} proves a localized converse
under formal equidimensionality. In the adic case,  Corollary~\ref{cor:swanson-full-specialization} recovers the full classical
localized converse of Rees
(\cite[Theorem~17.6.1]{SwansonHuneke}). As an application,
Theorem~\ref{thm:boger-graded-families} gives a B\"oger-type reduction
criterion for admissible graded families in terms of localized mixed
multiplicities \cite{Boger1969}. Its adic specialization,
Corollary~\ref{cor:boger-specialization}, recovers B\"oger's classical
theorem.

The paper is organized as follows.
Section~\ref{sec:preliminaries} recalls the classical theory of joint
reductions and mixed multiplicities, together with the corresponding
multiplicity theory for graded families and multigraded admissibility.
Section~\ref{sec:weak} introduces weak asymptotic joint reductions and
proves the asymptotic Rees theorem, including its adic and Noetherian
specializations.
Section~\ref{sec:strong} develops strong homogeneous joint reductions,
establishes their existence in the admissible setting, and proves the
weighted strong Rees theorem, together with its adic and Noetherian
specializations. We also compare the weak and strong notions.
Finally, Section~\ref{sec:converse} proves the localized converse for
strong homogeneous joint reductions, recovers the classical converse of
Rees in the adic case, and derives a B\"oger-type reduction theorem for
graded families. The failure of the converse in the weak asymptotic
setting is also exhibited.

\section{Preliminaries}
\label{sec:preliminaries}

Throughout, $(R,\mathfrak m)$ denotes a Noetherian local ring of dimension
$d\geq1$. For a finite $R$-module $M$ and an $\mathfrak m$-primary ideal
$I$, we write $e(I;M)$ for the Hilbert--Samuel multiplicity and
$e(I;R)=e(I)$.

\subsection*{Classical joint reductions and mixed multiplicities}

Let $J_1,\ldots,J_q$ be ideals of $R$ and choose $x_i\in J_i$. Following
\cite[Definition~17.1.1]{SwansonHuneke}, $(x_1,\ldots,x_q)$ is a
\emph{joint reduction} of $(J_1,\ldots,J_q)$ if
\(
\sum_{i=1}^{q}x_iJ_1\cdots\widehat{J_i}\cdots J_q
\)
is a reduction of $J_1\cdots J_q$. The corresponding module version is
\cite[Definition~17.1.3]{SwansonHuneke}. Let $I_1,\ldots,I_s$ be $\mathfrak m$-primary ideals and let $M$ be a
finite $R$-module. By \cite[Theorem~17.4.2]{SwansonHuneke}, for all
$n_1,\ldots,n_s\gg0$ the function
$\ell_R(M/I_1^{n_1}\cdots I_s^{n_s}M)$ agrees with a polynomial in
$n_1,\ldots,n_s$. In the convention of
\cite[Definition~17.4.3]{SwansonHuneke}, the homogeneous part of degree
$d=\dim R$ is
\[
\sum_{d_1+\cdots+d_s=d}
\frac{1}{d_1!\cdots d_s!}
 e(I_1^{[d_1]},\ldots,I_s^{[d_s]};M)
 n_1^{d_1}\cdots n_s^{d_s}.
\]
The coefficients are the classical mixed multiplicities. They are zero
when $\dim M<d$. The notation $I_i^{[d_i]}$ means that $I_i$ occurs
$d_i$ times and is unrelated to Frobenius powers.

We will repeatedly use Rees's fundamental theorem. If
$d_1+\cdots+d_s=d$ and the $d$ elements $x_{ij}\in I_i$ form a joint
reduction, with $I_i$ repeated $d_i$ times, then by \cite[Theorem~17.4.9]{SwansonHuneke},
\begin{equation}
\label{eq:classical-rees}
e(I_1^{[d_1]},\ldots,I_s^{[d_s]};M)=e((x_{ij});M).
\end{equation}

\subsection*{Graded families and mixed multiplicities}

A family $\mathcal I=\{I_n\}_{n\geq0}$ is a \emph{graded
$\mathfrak m$-primary family} if $I_0=R$, each $I_n$ is
$\mathfrak m$-primary for $n>0$, and $I_mI_n\subseteq I_{m+n}$ for all
$m,n\geq0$. For such a family, Cutkosky proves that
$\lim_{n\to\infty}e(I_n)/n^d$ exists
\cite[Theorem~1.3]{Cutkosky2026}.

Fix graded $\mathfrak m$-primary families
$\mathcal I^{(1)},\ldots,\mathcal I^{(s)}$, with
$\mathcal I^{(i)}=\{I_n^{(i)}\}_{n\geq0}$, and denote this collection by
$\boldsymbol{\mathcal I}$. For $\mathbf n=(n_1,\ldots,n_s)\in\N^s$, set
$|\mathbf n|=n_1+\cdots+n_s$,
$I_{\mathbf n}=I_{n_1}^{(1)}\cdots I_{n_s}^{(s)}$, and
$\mathbf t^{\mathbf n}=t_1^{n_1}\cdots t_s^{n_s}$. Cutkosky shows in \cite[Theorem~1.6]{Cutkosky2026} that there is a
homogeneous real polynomial $P_{\boldsymbol{\mathcal I}}$ of degree $d$
such that
\[
P_{\boldsymbol{\mathcal I}}(n_1,\ldots,n_s)
=
\lim_{m\to\infty}
\frac{e(I_{mn_1}^{(1)}\cdots I_{mn_s}^{(s)};R)}{m^d}.
\]

\begin{defn}
\label{def:mixed-family}
Following \cite[Theorem~1.6]{Cutkosky2026}, write
\[
P_{\boldsymbol{\mathcal I}}(\mathbf n)
=
\sum_{d_1+\cdots+d_s=d}
\frac{d!}{d_1!\cdots d_s!}
 e\!\left(\mathcal I^{(1)[d_1]},\ldots,
 \mathcal I^{(s)[d_s]}\right)
 n_1^{d_1}\cdots n_s^{d_s}.
\]
The coefficient
$e(\mathcal I^{(1)[d_1]},\ldots,\mathcal I^{(s)[d_s]})$ is the
\emph{mixed multiplicity of type $(d_1,\ldots,d_s)$}.
\end{defn}

Note that  Cutkosky's construction is compatible with the classical mixed
multiplicities at each level. More precisely,
\cite[Equation~(9)]{Cutkosky2026} gives
\begin{equation}
\label{eq:cutkosky-level-limit}
e\!\left(\mathcal I^{(1)[d_1]},\ldots,\mathcal I^{(s)[d_s]}\right)
=
\lim_{n\to\infty}
\frac{e(I_n^{(1)[d_1]},\ldots,I_n^{(s)[d_s]};R)}{n^d}
\end{equation}
for $d_1+\cdots+d_s=d$. In addition, mixed multiplicities with coefficients in a
finite module are constructed in \cite[Theorem~3.4]{Cutkosky2026}.

The following result gives the associativity and additivity formulas for
mixed multiplicities of graded families, extending the corresponding
classical formulas for mixed multiplicities of ideals. 
\begin{lem}
\label{lem:assoc-add}
Let $M$ be a finite $R$-module and let $d_1+\cdots+d_s=d$. Set
$\Lambda=\{P\in\Min(R)\mid\dim(R/P)=d\}$. If $\mathcal I^{(i)}(R/P)=\{(I_n^{(i)}+P)/P\}_{n\geq0}$, then
\[
e\!\left(\mathcal I^{(1)[d_1]},\ldots,\mathcal I^{(s)[d_s]};M\right)
=
\sum_{P\in\Lambda}\ell_{R_P}(M_P)\,
e\!\left(\mathcal I^{(1)}(R/P)^{[d_1]},\ldots,\mathcal I^{(s)}.(R/P)^{[d_s]}\right).
\]
Also, if $0\to M_1\to M_2\to M_3\to0$ is an exact sequence, one has
\[
e(\boldsymbol{\mathcal I}^{[\mathbf d]};M_2)
=
e(\boldsymbol{\mathcal I}^{[\mathbf d]};M_1)
+
e(\boldsymbol{\mathcal I}^{[\mathbf d]};M_3).
\]
\end{lem}

\begin{proof}
For a single graded family, Cutkosky's associativity formula with module
coefficients is \cite[Corollary~3.2]{Cutkosky2026}. Applying that formula
to the product family
$\{I_{mn_1}^{(1)}\cdots I_{mn_s}^{(s)}\}_{m\geq0}$ and using
\cite[Theorem~3.4]{Cutkosky2026}, we obtain an equality of homogeneous
polynomials in $n_1,\ldots,n_s$. Comparing the coefficient of
$n_1^{d_1}\cdots n_s^{d_s}$ gives the asserted associativity formula.
For a short exact sequence, localization at $P\in\Lambda$ gives $\ell_{R_P}((M_2)_P)=
\ell_{R_P}((M_1)_P)+\ell_{R_P}((M_3)_P).$
Substituting these equalities into the associativity formula gives
additivity.
\end{proof}

\subsection*{Multigraded admissibility}

Set $\mathcal F(\mathbf n)=I_{n_1}^{(1)}\cdots I_{n_s}^{(s)}$ for
$\mathbf n\in\N^s$. We use the standard notion of admissibility for
multigraded filtrations.

\begin{defn}
\label{def:admissible}
Let $\mathbf J=(J_1,\ldots,J_s)$ be $\mathfrak m$-primary ideals. We
say that $\boldsymbol{\mathcal I}$ is \emph{multigraded admissible with
respect to $\mathbf J$} if $\mathcal F$ extends to a
$\Z^s$-graded $\mathbf J$-admissible filtration. Explicitly, following
\cite[Definitions~1.3 and 1.4]{SarkarVerma}, we require, for all
$\mathbf m,\mathbf n\in\Z^s$,
\begin{enumerate}
\item $\mathbf J^{\mathbf n}\subseteq\mathcal F(\mathbf n)$;
\item $\mathcal F(\mathbf n)\mathcal F(\mathbf m)
      \subseteq\mathcal F(\mathbf n+\mathbf m)$;
\item if $\mathbf m\geq\mathbf n$, then
      $\mathcal F(\mathbf m)\subseteq\mathcal F(\mathbf n)$;
\item $\mathcal F(\mathbf n)=\mathcal F(\mathbf n^+)$, where
      $\mathbf n^+=(\max\{n_1,0\},\ldots,\max\{n_s,0\})$;
\item the extended Rees algebra $\mathcal R'(\mathcal F)
=
\bigoplus_{\mathbf n\in\Z^s}\mathcal F(\mathbf n)\mathbf t^{\mathbf n}$
is a finite module over
$\mathcal R'(\mathbf J)
=R[J_1t_1,\ldots,J_st_s,t_1^{-1},\ldots,t_s^{-1}]$.
\end{enumerate}
When the particular choice of the admissible base is irrelevant, we say
simply that $\boldsymbol{\mathcal I}$ is \emph{multigraded admissible}
if it is multigraded admissible with respect to some tuple
$\mathbf J=(J_1,\ldots,J_s)$ of $\mathfrak m$-primary ideals.
\end{defn}

It is important to realize that for an admissible filtration, the ordinary multigraded Rees algebra $\mathcal R(\boldsymbol{\mathcal I})
=
\mathcal R(\mathcal F)
=
\bigoplus_{\mathbf n\in\N^s}
\mathcal F(\mathbf n)\mathbf t^{\mathbf n}$
is a finite module over the standard multigraded algebra
$\mathcal R(\mathbf J)=R[J_1t_1,\ldots,J_st_s]$ by
\cite[Proposition~2.5]{MasutiSarkarVerma}. In particular,
$\mathcal R(\boldsymbol{\mathcal I})$ is Noetherian.

\section{Weak Joint Reductions and an Asymptotic Rees Theorem}
\label{sec:weak}

We begin by introducing the first generalized notion of joint reduction
for graded families. It is defined levelwise, by requiring classical
joint reductions at all sufficiently large levels.

\begin{defn}[Weak asymptotic joint reduction]
\label{def:weak}
Let $M$ be a finite $R$-module and let
$\mathbf r=(r_1,\ldots,r_s)\in\N^s$. A
\emph{weak asymptotic joint reduction of type $\mathbf r$ with respect
to $M$} is a collection
$\mathbf x_\bullet=\{\mathbf x(n)\}_{n\geq n_0}$, where
\[
\mathbf x(n)=
\{x_{ij}^{(n)}\mid 1\leq i\leq s,\ 1\leq j\leq r_i\},
\qquad
x_{ij}^{(n)}\in I_n^{(i)},
\]
such that, for every $n\geq n_0$, the tuple $\mathbf x(n)$ is a
classical joint reduction with respect to $M$ of
\[
\bigl(
\underbrace{I_n^{(1)},\ldots,I_n^{(1)}}_{r_1},
\ldots,
\underbrace{I_n^{(s)},\ldots,I_n^{(s)}}_{r_s}
\bigr).
\]
When $M=R$ we omit ``with respect to $M$''.
\end{defn}

The first question is whether weak asymptotic joint reductions actually
exist. Over an infinite residue field, the answer is affirmative.

\begin{thm}
\label{thm:weak-existence}
Assume that $R/\mathfrak m$ is infinite and let
$\mathbf r\in\N^s$ satisfy $|\mathbf r|=d$. Then
$\boldsymbol{\mathcal I}$ admits a weak asymptotic joint reduction of
type $\mathbf r$. 
\end{thm}

\begin{proof}
Fix $n\geq1$ and list $I_n^{(i)}$ exactly $r_i$ times. This gives $d$
$\mathfrak m$-primary ideals $J_1,\ldots,J_d$. Hence
$J_\ell\subseteq\sqrt{J_{\ell+1}\cdots J_d}=\mathfrak m$ for
$\ell<d$. The theorem
\cite[Theorem~17.3.1]{SwansonHuneke} gives elements
$y_\ell\in J_\ell$ forming a joint reduction. Regrouping the $y_\ell$
according to the repeated ideals gives the required
$x_{ij}^{(n)}$. Since $n$ was arbitrary, the choices form a weak
asymptotic joint reduction.
\end{proof}

We now pass from the ring case to finite modules and record the
corresponding levelwise asymptotic realization of mixed multiplicities.

\begin{lem}[Asymptotic realization with coefficients]
\label{lem:module-level-limit}
Let $M$ be a finite $R$-module with $\dim M=d$ and
$d_1+\cdots+d_s=d$. Then
\[
e\!\left(\mathcal I^{(1)[d_1]},\ldots,
\mathcal I^{(s)[d_s]};M\right)
=
\lim_{n\to\infty}
\frac{
e(I_n^{(1)[d_1]},\ldots,I_n^{(s)[d_s]};M)
}{n^d}.
\]
\end{lem}

\begin{proof}
By Lemma~\ref{lem:assoc-add}, the mixed multiplicity on $M$ is a finite
sum over the maximal-dimensional minimal components of $R$. On each
component, \cite[Equation~(9)]{Cutkosky2026} gives the asserted
normalized limit. Passing through the finite sum and using the classical
associativity formula for mixed multiplicities
\cite[Lemma~17.4.4]{SwansonHuneke} yields the desired  formula.
\end{proof}

\begin{thm}[Asymptotic Rees theorem]
\label{thm:weak-rees}
Let $M$ be a finite $R$-module with $\dim M=d$, let
$\mathbf d=(d_1,\ldots,d_s)\in\N^s$ satisfy $|\mathbf d|=d$, and let
$\mathbf x_\bullet=\{\mathbf x(n)\}_{n\gg0}$ be a weak asymptotic joint
reduction of type $\mathbf d$ with respect to $M$. Put $Q_n=(x_{ij}^{(n)}\mid 1\leq i\leq s,\ 1\leq j\leq d_i).$
Then
\[
e\!\left(\mathcal I^{(1)[d_1]},\ldots,
\mathcal I^{(s)[d_s]};M\right)
=
\lim_{n\to\infty}\frac{e(Q_n;M)}{n^d}.
\]
\end{thm}

\begin{proof}
For every sufficiently large $n$,
$e(I_n^{(1)[d_1]},\ldots,I_n^{(s)[d_s]};M)=e(Q_n;M)$ by  Rees's theorem
\cite[Theorem~17.4.9]{SwansonHuneke}.
Divide by $n^d$ and pass to the limit using
Lemma~\ref{lem:module-level-limit}.
\end{proof}

\begin{cor}
\label{cor:weak-realization}
If $R/\mathfrak m$ is infinite, then for every
$\mathbf d\in\N^s$ with $|\mathbf d|=d$ there are $d$-generated ideals
$Q_n$ such that
\[
e\!\left(\mathcal I^{(1)[d_1]},\ldots,
\mathcal I^{(s)[d_s]}\right)
=
\lim_{n\to\infty}\frac{e(Q_n;R)}{n^d}.
\]
\end{cor}

\begin{proof}
The result is a consequence of  Theorems~\ref{thm:weak-existence} and~\ref{thm:weak-rees}.
\end{proof}

\subsection*{Adic and Noetherian specializations}

We first record the specialization to adic families. In this case the
mixed multiplicities of graded families recover the classical mixed
multiplicities, and Rees's theorem
\cite[Theorem~17.4.9]{SwansonHuneke} applies at every level.

\begin{cor}
\label{cor:weak-rees-adic}
Let $(R,\mathfrak m)$ be a $d$-dimensional Noetherian local ring, let
$M$ be a finite $R$-module with $\dim M=d$, and let
$J_1,\ldots,J_s$ be $\mathfrak m$-primary ideals. For each $i$, consider
the $J_i$-adic family $\mathcal I^{(i)}
=
\{J_i^n\}_{n\geq0}.$
Let $\mathbf d=(d_1,\ldots,d_s)\in\mathbb N^s$ satisfy
$d_1+\cdots+d_s=d$, and let
$\mathbf x_\bullet=\{\mathbf x(n)\}_{n\gg0}$ be a weak asymptotic joint
reduction of type $\mathbf d$ with respect to $M$. Set $Q_n
=
(x_{ij}^{(n)}
\mid
1\leq i\leq s,\;
1\leq j\leq d_i).$
Then, 
\[
e\!\left(
\mathcal I^{(1)[d_1]},\ldots,
\mathcal I^{(s)[d_s]};M
\right)
=
e\!\left(
J_1^{[d_1]},\ldots,J_s^{[d_s]};M
\right).
\]
\end{cor}

\begin{proof}
For every sufficiently large $n$, the tuple $\mathbf x(n)$ is a
classical joint reduction, with respect to $M$, of the $d$ ideals
consisting of $J_i^n$ repeated $d_i$ times. Therefore
\cite[Theorem~17.4.9]{SwansonHuneke} gives
\[
e(Q_n;M)
=
e\!\left(
(J_1^n)^{[d_1]},\ldots,
(J_s^n)^{[d_s]};M
\right).
\]
By the homogeneity of classical mixed multiplicities,
\cite[Theorem~17.4.2 and Definition~17.4.3]{SwansonHuneke},
\[
e\!\left(
(J_1^n)^{[d_1]},\ldots,
(J_s^n)^{[d_s]};M
\right)
=
n^{d_1+\cdots+d_s}
e\!\left(
J_1^{[d_1]},\ldots,J_s^{[d_s]};M
\right).
\]
Since $d_1+\cdots+d_s=d$, we derive $e(Q_n;M)
=
n^d
e\!\left(
J_1^{[d_1]},\ldots,J_s^{[d_s]};M
\right).$
Dividing by $n^d$ gives the second assertion. Passing to the limit and
using Theorem~\ref{thm:weak-rees}, we obtain the result.
\end{proof}


We next specialize to Noetherian graded families. A sufficiently
divisible common Veronese is standard
\cite[Theorem~2.1]{HerzogHibiTrung}, so the families become adic along
that Veronese. Thus Theorem~\ref{thm:weak-rees} reduces there to the
classical Rees theorem.

\begin{cor}
\label{cor:weak-rees-noetherian}
Let $(R,\mathfrak m)$ be a $d$-dimensional Noetherian local ring and let
$M$ be a finite $R$-module with $\dim M=d$. Let $\boldsymbol{\mathcal I}
=
(\mathcal I^{(1)},\ldots,\mathcal I^{(s)})$
be graded $\mathfrak m$-primary families, and assume that each
$\mathcal I^{(i)}$ is Noetherian.
Let $\mathbf d=(d_1,\ldots,d_s)\in\mathbb N^s$ satisfy
$d_1+\cdots+d_s=d$, and let
$\mathbf x_\bullet=\{\mathbf x(n)\}_{n\gg0}$ be a weak asymptotic joint
reduction of type $\mathbf d$ with respect to $M$. Set $Q_n
=
(x_{ij}^{(n)}
\mid
1\leq i\leq s,\;
1\leq j\leq d_i).$
Then there exists an integer $c\geq1$ such that, for every
$1\leq i\leq s$ and every $m\geq0$, $I_{cm}^{(i)}
=
\bigl(I_c^{(i)}\bigr)^m.$
For every sufficiently large $m$ one has
\[
e(Q_{cm};M)
=
m^d
e\!\left(
I_c^{(1)[d_1]},\ldots,
I_c^{(s)[d_s]};M
\right).
\]
Consequently,
\[
e\!\left(
\mathcal I^{(1)[d_1]},\ldots,
\mathcal I^{(s)[d_s]};M
\right)
=
\frac{1}{c^d}
e\!\left(
I_c^{(1)[d_1]},\ldots,
I_c^{(s)[d_s]};M
\right),
\]
and, more precisely, for every sufficiently large $m$, one has
\[
\frac{e(Q_{cm};M)}{(cm)^d}
=
e\!\left(
\mathcal I^{(1)[d_1]},\ldots,
\mathcal I^{(s)[d_s]};M
\right).
\]

\end{cor}

\begin{proof}
For each $i$, the Noetherianity of $\mathcal I^{(i)}$ means that $\mathcal R(\mathcal I^{(i)})
=
\bigoplus_{n\geq0}I_n^{(i)}t^n$
is a finitely generated positively graded $R$-algebra. By
\cite[Theorem~2.1]{HerzogHibiTrung}, there exists a positive integer
$c_i$ such that the $c_i$-th Veronese subalgebra $\mathcal R(\mathcal I^{(i)})^{(c_i)}
=
\bigoplus_{m\geq0}
I_{c_im}^{(i)}t^{c_im}$
is standard after regrading. Equivalently, $I_{c_im}^{(i)}
=
\bigl(I_{c_i}^{(i)}\bigr)^m$
for every $m\geq0$. Choose a common positive multiple $c$ of $c_1,\ldots,c_s$, sufficiently
divisible so that the $c$-th Veronese of each
$\mathcal R(\mathcal I^{(i)})$ is standard. Then $I_{cm}^{(i)}
=
\bigl(I_c^{(i)}\bigr)^m$
for all $i$ and all $m\geq0$.

Now fix $m\gg0$. Since $\mathbf x_\bullet$ is a weak asymptotic joint
reduction, $\mathbf x(cm)$ is a classical joint reduction of the list
in which $I_{cm}^{(i)}$ occurs $d_i$ times. Using the Veronese
identities, this is the list
\[
\underbrace{(I_c^{(1)})^m,\ldots,(I_c^{(1)})^m}_{d_1},
\ldots,
\underbrace{(I_c^{(s)})^m,\ldots,(I_c^{(s)})^m}_{d_s}.
\]
Thus Rees's theorem
\cite[Theorem~17.4.9]{SwansonHuneke} yields
\[
e(Q_{cm};M)
=
e\!\left(
((I_c^{(1)})^m)^{[d_1]},\ldots,
((I_c^{(s)})^m)^{[d_s]};M
\right).
\]
By homogeneity of the classical mixed multiplicities,
\cite[Theorem~17.4.2 and Definition~17.4.3]{SwansonHuneke},
the right-hand side equals $m^{d_1+\cdots+d_s}
e\!\left(
I_c^{(1)[d_1]},\ldots,
I_c^{(s)[d_s]};M
\right).$
Since $d_1+\cdots+d_s=d$, we obtain
\[
e(Q_{cm};M)
=
m^d
e\!\left(
I_c^{(1)[d_1]},\ldots,
I_c^{(s)[d_s]};M
\right).
\]

Dividing by $(cm)^d$ gives
\[
\frac{e(Q_{cm};M)}{(cm)^d}
=
\frac{1}{c^d}
e\!\left(
I_c^{(1)[d_1]},\ldots,
I_c^{(s)[d_s]};M
\right).
\]
On the other hand, Theorem~\ref{thm:weak-rees} provides
\[
e\!\left(
\mathcal I^{(1)[d_1]},\ldots,
\mathcal I^{(s)[d_s]};M
\right)
=
\lim_{n\to\infty}\frac{e(Q_n;M)}{n^d}.
\]
Taking the subsequence $n=cm$ and using the preceding equality gives
\[
e\!\left(
\mathcal I^{(1)[d_1]},\ldots,
\mathcal I^{(s)[d_s]};M
\right)
=
\frac{1}{c^d}
e\!\left(
I_c^{(1)[d_1]},\ldots,
I_c^{(s)[d_s]};M
\right).
\]
Substituting this identity back into the preceding formula proves the
last assertion.
\end{proof}

\section{Strong Homogeneous Joint Reductions and Exact Rees Formulas}
\label{sec:strong}

We now introduce a second notion of joint reduction for graded families.
In contrast with the weak levelwise notion, it is defined directly in
the multigraded Rees algebra and leads to an exact Rees-type multiplicity
formula.

\begin{defn}
\label{def:strong}
Let $\mathbf r=(r_1,\ldots,r_s)\in\N^s$. A \emph{strong homogeneous
joint reduction of type $\mathbf r$} of $\boldsymbol{\mathcal I}$
consists of positive integers $a_{ij}$ and elements
$x_{ij}\in I_{a_{ij}}^{(i)}$ such that, after setting
$X_{ij}=x_{ij}t_i^{a_{ij}}$, one has
\[
\mathcal R(\boldsymbol{\mathcal I})_{\mathbf n}
=
\sum_{i=1}^{s}\sum_{j=1}^{r_i}
X_{ij}\,
\mathcal R(\boldsymbol{\mathcal I})_{\mathbf n-a_{ij}\mathbf e_i}
\qquad(\mathbf n\gg\mathbf0).
\]
Equivalently, for all sufficiently large $\mathbf n$,
\[
I_{n_1}^{(1)}\cdots I_{n_s}^{(s)}
=
\sum_{i=1}^{s}\sum_{j=1}^{r_i}
 x_{ij}
 I_{n_1}^{(1)}\cdots I_{n_i-a_{ij}}^{(i)}\cdots I_{n_s}^{(s)}.
\]

\end{defn}

\begin{prop}
\label{prop:intrinsic}
Let $\mathcal A=\mathcal R(\boldsymbol{\mathcal I})$ and
$\mathcal X=(X_{ij})\mathcal A$, where
$X_{ij}=x_{ij}t_i^{a_{ij}}$. Then $\{x_{ij}\}$ is a strong homogeneous
joint reduction of type $\mathbf r$ if and only if
$(\mathcal A/\mathcal X)_{\mathbf n}=0$ for all
$\mathbf n\gg\mathbf0$.
\end{prop}

\begin{proof}
Since the $X_{ij}$ are homogeneous, $\mathcal X$ is homogeneous and $\mathcal X_{\mathbf n}
=
\sum_{i,j}X_{ij}\mathcal A_{\mathbf n-a_{ij}\mathbf e_i}.$
Hence the defining equality in Definition~\ref{def:strong} is equivalent
to $\mathcal X_{\mathbf n}=\mathcal A_{\mathbf n}$ for every
$\mathbf n\gg\mathbf0$, which in turn is equivalent to the stated
vanishing of the homogeneous components of the quotient.
\end{proof}

\subsection*{Admissible collections}

We first collect several facts about admissible collections that will be
used for the rest of this section.

\begin{lem}
\label{lem:admissible-invariance}
Assume that $\boldsymbol{\mathcal I}$ is multigraded admissible with
respect to $\mathbf J=(J_1,\ldots,J_s)$. Then, for every
$d_1+\cdots+d_s=d$,
\[
e\!\left(\mathcal I^{(1)[d_1]},\ldots,
\mathcal I^{(s)[d_s]}\right)
=
e(J_1^{[d_1]},\ldots,J_s^{[d_s]}).
\]
\end{lem}

\begin{proof}
Write $\mathcal A=\mathcal R(\mathcal F)$ and
$\mathcal B=\mathcal R(\mathbf J)$. By
\cite[Proposition~2.5]{MasutiSarkarVerma}, $\mathcal A$ is a finite
graded $\mathcal B$-module. Choose homogeneous $\mathcal B$-module
generators
$z_1,\ldots,z_v$ of $\mathcal A$, say
$z_\ell=f_\ell\mathbf t^{\mathbf b_\ell}$ with
$f_\ell\in\mathcal F(\mathbf b_\ell)$. Choose
$\mathbf c\in\N^s$ with $\mathbf b_\ell\leq\mathbf c$ componentwise
for every $\ell$. For $\mathbf n\gg\mathbf0$, taking the homogeneous component of degree
$\mathbf n$ in the equality
$\mathcal A=\sum_\ell\mathcal Bz_\ell$ gives $\mathcal F(\mathbf n)
=
\sum_{\ell=1}^{v}
\mathbf J^{\mathbf n-\mathbf b_\ell}f_\ell.$
Since each $f_\ell\in R$, every summand is contained in
$\mathbf J^{\mathbf n-\mathbf b_\ell}$. Moreover,
$\mathbf b_\ell\leq\mathbf c$ implies
$\mathbf n-\mathbf b_\ell\geq\mathbf n-\mathbf c$, hence
$\mathbf J^{\mathbf n-\mathbf b_\ell}
\subseteq\mathbf J^{\mathbf n-\mathbf c}$. Therefore $\mathcal F(\mathbf n)\subseteq\mathbf J^{\mathbf n-\mathbf c}$
 for $\mathbf n\gg\mathbf0).$
On the other hand, because $\mathcal F$ is a $\mathbf J$-filtration,
Definition~\ref{def:admissible} gives
$\mathbf J^{\mathbf n}\subseteq\mathcal F(\mathbf n)$. Consequently,
\begin{equation}
\label{eq:admissible-sandwich}
\mathbf J^{\mathbf n}
\subseteq\mathcal F(\mathbf n)
\subseteq\mathbf J^{\mathbf n-\mathbf c}
\qquad(\mathbf n\gg\mathbf0).
\end{equation}

Fix $\mathbf q=(q_1,\ldots,q_s)\in\N_{>0}^s$. For $m\gg0$,
\eqref{eq:admissible-sandwich} gives $\mathbf J^{m\mathbf q}
\subseteq\mathcal F(m\mathbf q)
\subseteq\mathbf J^{m\mathbf q-\mathbf c}.$
Since the Hilbert--Samuel multiplicity is decreasing under inclusion of
$\mathfrak m$-primary ideals, we derive
\[
e(\mathbf J^{m\mathbf q})
\geq e(\mathcal F(m\mathbf q))
\geq e(\mathbf J^{m\mathbf q-\mathbf c}).
\]
The first outer term satisfies
$e(\mathbf J^{m\mathbf q})=m^d e(J_1^{q_1}\cdots J_s^{q_s})$.
For the second outer term, apply
\cite[Theorem~1.6]{Cutkosky2026} to the adic families
$\{J_i^n\}_{n\geq0}$. The function
$e(J_1^{u_1}\cdots J_s^{u_s})$ is thereby a homogeneous polynomial of
degree $d$ in $u_1,\ldots,u_s$, and hence
\[
\lim_{m\to\infty}
\frac{e(\mathbf J^{m\mathbf q-\mathbf c})}{m^d}
=
e(J_1^{q_1}\cdots J_s^{q_s}).
\]
Dividing the preceding inequalities by $m^d$ and applying the squeeze
theorem provides that $P_{\boldsymbol{\mathcal I}}(\mathbf q)
=
e(J_1^{q_1}\cdots J_s^{q_s}).$
Since both sides are homogeneous polynomials of degree $d$ and agree for
all $\mathbf q\in\mathbb N_{>0}^s$, they are identical. The result follows comparing the
coefficients of $q_1^{d_1}\cdots q_s^{d_s}$.
\end{proof}

\begin{lem}
\label{lem:change-base}
Let $\mathcal F$ be $\mathbf J$-admissible and set
$K_i=\mathcal F(\mathbf e_i)$. Then $\mathcal F$ is also
$\mathbf K=(K_1,\ldots,K_s)$-admissible.
\end{lem}

\begin{proof}
Since $\mathcal F$ is a $\mathbf J$-filtration,
$J_i\subseteq K_i$. By multiplicativity,
$K_1^{n_1}\cdots K_s^{n_s}\subseteq\mathcal F(\mathbf n)$ for every
$\mathbf n\in\N^s$, while the monotonicity and the equality
$\mathcal F(\mathbf n)=\mathcal F(\mathbf n^+)$ are unchanged. Moreover, $\mathcal R'(\mathbf J)
\subseteq\mathcal R'(\mathbf K)
\subseteq\mathcal R'(\mathcal F).$
If $u_1,\ldots,u_v$ generate $\mathcal R'(\mathcal F)$ over
$\mathcal R'(\mathbf J)$, then the same elements generate it over the
larger intermediate algebra $\mathcal R'(\mathbf K)$. Therefore
$\mathcal F$ is $\mathbf K$-admissible.
\end{proof}

\begin{lem}
\label{lem:veronese-admissible}
Assume that $\mathcal F$ is $\mathbf J$-admissible and let
$\mathbf A=(A_1,\ldots,A_s)\in\N_{>0}^s$. Define $\mathcal F^{(\mathbf A)}(\mathbf n)
=
\mathcal F(A_1n_1,\ldots,A_sn_s).$
Then $\mathcal F^{(\mathbf A)}$ is
$\mathbf J^{(\mathbf A)}=(J_1^{A_1},\ldots,J_s^{A_s})$-admissible.
\end{lem}

\begin{proof}
The filtration, multiplicativity, monotonicity and positive-part
conditions follow immediately from those of $\mathcal F$. It remains to
prove module-finiteness of the extended Rees algebra. Put $B=\mathcal R'(\mathbf J)$ and
$C=\mathcal R'(\mathcal F)$. By admissibility, $C$ is finite over $B$.
Let $B^{(\mathbf A)}$ and $C^{(\mathbf A)}$ denote the coordinatewise
Veronese subrings, so that
$C^{(\mathbf A)}=\mathcal R'(\mathcal F^{(\mathbf A)})$ and
$B^{(\mathbf A)}=\mathcal R'(\mathbf J^{(\mathbf A)})$ after the
obvious regrading. A Veronese submodule of a finite multigraded module is
finite over the corresponding Veronese subring. For completeness, choose
homogeneous $B$-module generators $u_1,\ldots,u_v$ of $C$. Splitting
these generators and the coefficients according to the finitely many
residue classes in
$\prod_i\mathbb Z/A_i\mathbb Z$ shows that each Veronese component of
$C$ is generated over $B^{(\mathbf A)}$ by finitely many homogeneous
pieces. Hence $C^{(\mathbf A)}$ is finite over $B^{(\mathbf A)}$.
\end{proof}

In the admissible setting, strong homogeneous joint reductions exist
already in degree one. This will also
give existence in arbitrary prescribed positive degrees.

\begin{thm}
\label{thm:strong-existence}
Assume that $R/\mathfrak m$ is infinite and that
$\boldsymbol{\mathcal I}$ is multigraded admissible with respect to
$\mathbf J=(J_1,\ldots,J_s)$. If
$\mathbf r=(r_1,\ldots,r_s)\in\N^s$ satisfies $|\mathbf r|=d$, then
there exist $x_{ij}\in J_i\subseteq I_1^{(i)}$ such that $\mathcal F(\mathbf n)
=
\sum_{i=1}^{s}\sum_{j=1}^{r_i}
 x_{ij}\mathcal F(\mathbf n-\mathbf e_i)$ for  $\mathbf n\gg\mathbf0$. In particular, $\{x_{ij}\}$ is a degree-one strong homogeneous joint
reduction of type $\mathbf r$.
\end{thm}

\begin{proof}
By \cite[Theorem~2.4]{SarkarVerma}, an admissible filtration over a local
ring of dimension $d\geq1$ with infinite residue field admits a joint
reduction of every type $\mathbf r$ with $|\mathbf r|=d$, consisting of
elements $x_{ij}\in J_i$. By
\cite[Definition~1.6]{SarkarVerma}, this means precisely that the
displayed equality holds for every sufficiently large multidegree.
Since $J_i\subseteq\mathcal F(\mathbf e_i)=I_1^{(i)}$, the elements
$x_{ij}t_i$ belong to $\mathcal R(\boldsymbol{\mathcal I})$ and satisfy
Definition~\ref{def:strong} with all homogeneous degrees equal to one.
\end{proof}

\begin{cor}
\label{cor:weighted-strong-existence}
Assume that $R/\mathfrak m$ is infinite and that
$\boldsymbol{\mathcal I}$ is multigraded admissible with respect to
$\mathbf J=(J_1,\ldots,J_s)$. Let
$\mathbf r=(r_1,\ldots,r_s)\in\mathbb N^s$ satisfy $|\mathbf r|=d$,
and prescribe positive integers $a_{ij}$ for
$1\leq i\leq s$ and $1\leq j\leq r_i$. Then there exist elements $y_{ij}\in J_i^{a_{ij}}\subseteq I_{a_{ij}}^{(i)}$
forming a strong homogeneous joint reduction of type $\mathbf r$ with
homogeneous degrees $a_{ij}$.
\end{cor}

\begin{proof}
By Theorem~\ref{thm:strong-existence}, there exist
$x_{ij}\in J_i\subseteq I_1^{(i)}$ forming a degree-one strong
homogeneous joint reduction of type $\mathbf r$. Set
$y_{ij}=x_{ij}^{a_{ij}}$. Then
$y_{ij}\in J_i^{a_{ij}}\subseteq I_{a_{ij}}^{(i)}$, and
Lemma~\ref{lem:power-stability} shows that $\{y_{ij}\}$ remains a strong
homogeneous joint reduction, now with homogeneous degrees $a_{ij}$.
\end{proof}

The degree-one case provides the basic exact Rees formula in the strong
setting and will serve as the starting point for the weighted version. 
\begin{prop}
\label{prop:degree-one-rees}
Assume that $\boldsymbol{\mathcal I}$ is multigraded admissible. Let
$\mathbf d=(d_1,\ldots,d_s)\in\N^s$ satisfy $|\mathbf d|=d$, and suppose
that $x_{ij}\in I_1^{(i)}$ form a degree-one strong homogeneous joint
reduction of type $\mathbf d$. Set $Q=(x_{ij})$. Then $Q$ is
$\mathfrak m$-primary and
\[
e(Q;R)
=
e\!\left(\mathcal I^{(1)[d_1]},\ldots,
\mathcal I^{(s)[d_s]}\right).
\]
\end{prop}

\begin{proof}
Let $\mathcal F$ be the associated multigraded admissible filtration and
put $K_i=\mathcal F(\mathbf e_i)=I_1^{(i)}$. By
Lemma~\ref{lem:change-base}, $\mathcal F$ is $\mathbf K$-admissible.
Set
$\mathcal A=\mathcal R(\mathcal F)$ and
$\mathcal B=\mathcal R(\mathbf K)=R[K_1t_1,\ldots,K_st_s]$.
Then $\mathcal A$ is a finite graded $\mathcal B$-module, hence
$\mathcal B\subseteq\mathcal A$ is integral.

Put $X_{ij}=x_{ij}t_i\in\mathcal B$,
$L=(X_{ij})\mathcal B$, and $L\mathcal A=(X_{ij})\mathcal A$.
Since the reduction is strong,
Proposition~\ref{prop:intrinsic} gives
$(\mathcal A/L\mathcal A)_{\mathbf n}=0$ for all
$\mathbf n\gg\mathbf0$. Let $\mathcal B_{++}
=
\bigoplus_{\mathbf n\geq\mathbf e}\mathcal B_{\mathbf n},$ where $\mathbf e=(1,\ldots,1).$
The quotient $C=\mathcal A/L\mathcal A$ is a finite graded
$\mathcal B$-module. Now, choose homogeneous generators
$\overline u_1,\ldots,\overline u_v$ of degrees
$\mathbf b_1,\ldots,\mathbf b_v$. Choose
$\mathbf m$ such that $C_{\mathbf n}=0$ for all
$\mathbf n\geq\mathbf m$, and then choose $N$ so large that
$\mathbf b_\ell+N\mathbf e\geq\mathbf m$ for every $\ell$. Every
homogeneous element of $\mathcal B_{++}^N\overline u_\ell$ has degree
at least $\mathbf b_\ell+N\mathbf e$, hence is zero in $C$. Therefore $\mathcal B_{++}^{N}C=0$, or equivalently,
$\mathcal B_{++}^{N}\mathcal A\subseteq L\mathcal A.$
Applying this inclusion to $1\in\mathcal A$ gives
$\mathcal B_{++}^{N}\subseteq L\mathcal A\cap\mathcal B$. Since
$\mathcal A$ is integral over $\mathcal B$, extension and contraction of
ideals give $\sqrt{L\mathcal A\cap\mathcal B}=\sqrt L.$
Consequently $\mathcal B_{++}\subseteq\sqrt L$. As $\mathcal B$ is
Noetherian, there exists $N'>0$ with
$\mathcal B_{++}^{N'}\subseteq L$.

If $\mathbf n\geq N'\mathbf e$, standardness of the multi-Rees algebra
implies
$\mathcal B_{\mathbf n}
=\mathcal B_{N'\mathbf e}
 \mathcal B_{\mathbf n-N'\mathbf e}$, while
$\mathcal B_{N'\mathbf e}\subseteq\mathcal B_{++}^{N'}$. Hence
$\mathcal B_{\mathbf n}\subseteq L$, and therefore $\mathcal B_{\mathbf n}=L_{\mathbf n}
=
\sum_{i=1}^{s}\sum_{j=1}^{d_i}
 (x_{ij}t_i)\mathcal B_{\mathbf n-\mathbf e_i}$ for  $\mathbf n\gg\mathbf0.$
Cancelling $\mathbf t^{\mathbf n}$ gives
\begin{equation}
\label{eq:K-tail}
K_1^{n_1}\cdots K_s^{n_s}
=
\sum_{i=1}^{s}\sum_{j=1}^{d_i}
 x_{ij}K_1^{n_1}\cdots K_i^{n_i-1}\cdots K_s^{n_s}
\end{equation}
for every sufficiently large $\mathbf n$.
We now extract a classical joint reduction. Let
$S=\{i\mid d_i>0\}$ and $T=\{i\mid d_i=0\}$. Choose integers
$b_h\gg0$ for $h\in T$ so that
\eqref{eq:K-tail} applies whenever the inactive coordinates equal
$b_h$, and set
$M_0=\prod_{h\in T}K_h^{b_h}$, with $M_0=R$ if $T=\varnothing$.
Put $P=\prod_{i\in S}K_i^{d_i}$, and $H=
\sum_{i\in S}\sum_{j=1}^{d_i}
 x_{ij}K_i^{d_i-1}
 \prod_{\substack{h\in S\\h\neq i}}K_h^{d_h}.$
Choosing $r\gg0$ and substituting
$n_i=(r+1)d_i$ for $i\in S$ and $n_h=b_h$ for $h\in T$ in
\eqref{eq:K-tail} gives
$P^{r+1}M_0=HP^rM_0$. Thus the $x_{ij}$ form a classical joint reduction,
with respect to the finite module $M_0$, of the list in which $K_i$ is
repeated $d_i$ times.

The ideal $M_0$ is either $R$ or an $\mathfrak m$-primary ideal. In the
latter case $R/M_0$ has dimension zero. Since $d\geq1$, additivity of
ordinary multiplicity and of mixed multiplicities
\cite[Theorem~11.2.4 and Lemma~17.4.4]{SwansonHuneke} gives $e(Q;M_0)=e(Q;R),$ and $e(K_1^{[d_1]},\ldots,K_s^{[d_s]};M_0)
 =e(K_1^{[d_1]},\ldots,K_s^{[d_s]};R).$
Applying Rees's theorem \cite[Theorem~17.4.9]{SwansonHuneke} to the
joint reduction with respect to $M_0$ yields $e(Q;R)=e(K_1^{[d_1]},\ldots,K_s^{[d_s]};R).$
In particular $Q$ is $\mathfrak m$-primary. Finally,
Lemma~\ref{lem:admissible-invariance}, applied to the admissible base
$\mathbf K$, identifies the right-hand side with the mixed multiplicity
of the graded families.
\end{proof}

\begin{lem}
\label{lem:power-stability}
Assume that $\boldsymbol{\mathcal I}$ is multigraded admissible and let
$X_1,\ldots,X_q\in\mathcal A=\mathcal R(\boldsymbol{\mathcal I})$ be
the homogeneous elements associated to a strong homogeneous joint
reduction. For arbitrary positive integers $c_1,\ldots,c_q$, the
homogeneous elements $X_1^{c_1},\ldots,X_q^{c_q}$ also define a strong
homogeneous joint reduction, with the corresponding degrees multiplied
by $c_1,\ldots,c_q$.
\end{lem}

\begin{proof}
Choose an admissible base $\mathbf J$ and put
$\mathcal B=\mathcal R(\mathbf J)$. Then $\mathcal A$ is a finite
$\mathcal B$-module. Set
$L=(X_1,\ldots,X_q)\mathcal A$ and
$L'=(X_1^{c_1},\ldots,X_q^{c_q})\mathcal A$. By the strong hypothesis
and Proposition~\ref{prop:intrinsic}, $(\mathcal A/L)_{\mathbf n}=0$
for $\mathbf n\gg\mathbf0$.

Arguing exactly as in the proof of
Proposition~\ref{prop:degree-one-rees} with the finite
$\mathcal B$-module $\mathcal A/L$, there exists $N$ such that
$\mathcal B_{++}^{N}\mathcal A\subseteq L$. Hence
$\mathcal B_{++}\mathcal A\subseteq\sqrt L$. Since
\[
\sqrt{L'}=
\sqrt{(X_1^{c_1},\ldots,X_q^{c_q})\mathcal A}
=
\sqrt{(X_1,\ldots,X_q)\mathcal A}
=
\sqrt L,
\]
we also have $\mathcal B_{++}\mathcal A\subseteq\sqrt{L'}$. The ring
$\mathcal A$ is Noetherian, so for some $N'>0$,
$(\mathcal B_{++}\mathcal A)^{N'}\subseteq L'$, and therefore
$\mathcal B_{++}^{N'}\mathcal A\subseteq L'$.
Now, choose homogeneous $\mathcal B$-module generators
$u_1,\ldots,u_v$ of $\mathcal A$, with degrees $\mathbf b_\ell$.
For $\mathbf n\gg\mathbf0$ we have
$\mathbf n-\mathbf b_\ell\geq N'\mathbf e$ for every $\ell$.
Standardness of $\mathcal B$ then gives
$\mathcal B_{\mathbf n-\mathbf b_\ell}
\subseteq\mathcal B_{++}^{N'}$, whence $\mathcal A_{\mathbf n}
=
\sum_\ell\mathcal B_{\mathbf n-\mathbf b_\ell}u_\ell
\subseteq
\mathcal B_{++}^{N'}\mathcal A
\subseteq L'.$
Thus $(\mathcal A/L')_{\mathbf n}=0$ for all
$\mathbf n\gg\mathbf0$. Proposition~\ref{prop:intrinsic} completes the
proof.
\end{proof}

We now pass from the degree-one formula to arbitrary homogeneous
degrees. The next theorem gives the weighted form of the strong Rees
formula and is the main multiplicity result of this section.

\begin{thm}
\label{thm:weighted-strong-rees}
Let $(R,\mathfrak m)$ be a $d$-dimensional Noetherian local ring with
$d\geq1$, and let
$\boldsymbol{\mathcal I}=(\mathcal I^{(1)},\ldots,\mathcal I^{(s)})$
be a multigraded admissible collection of graded
$\mathfrak m$-primary families. Let
$\mathbf d=(d_1,\ldots,d_s)\in\N^s$ satisfy $d_1+\cdots+d_s=d$.
Suppose that $\mathcal X=
\{x_{ij}\mid 1\leq i\leq s,\ 1\leq j\leq d_i\}$
is a strong homogeneous joint reduction of type $\mathbf d$, where
$x_{ij}\in I_{a_{ij}}^{(i)}$ and $a_{ij}\geq1$. Set
$Q=(x_{ij}\mid 1\leq i\leq s,\ 1\leq j\leq d_i)$. Then $Q$ is
$\mathfrak m$-primary and
\[
e(Q;R)
=
\left(\prod_{i=1}^{s}\prod_{j=1}^{d_i}a_{ij}\right)
e\!\left(\mathcal I^{(1)[d_1]},\ldots,
\mathcal I^{(s)[d_s]}\right).
\]
\end{thm}

\begin{proof}
For every $i$ with $d_i>0$, let
$A_i=\operatorname{lcm}(a_{i1},\ldots,a_{id_i})$ and set
$c_{ij}=A_i/a_{ij}$. If $d_i=0$, put $A_i=1$. Define
$y_{ij}=x_{ij}^{c_{ij}}$. Since the families are graded, $y_{ij}\in(I_{a_{ij}}^{(i)})^{c_{ij}}
\subseteq I_{A_i}^{(i)}.$
Let $X_{ij}=x_{ij}t_i^{a_{ij}}$ and
$Y_{ij}=y_{ij}t_i^{A_i}=X_{ij}^{c_{ij}}$. By
Lemma~\ref{lem:power-stability}, the elements $Y_{ij}$ define a strong
homogeneous joint reduction of the original collection, with common
degree $A_i$ in the $i$-th direction.

For each $i$, define the coordinatewise Veronese family $\mathcal G^{(i)}$ by
$G_n^{(i)}=I_{A_in}^{(i)}$ for $n\geq0$. By
Lemma~\ref{lem:veronese-admissible}, the collection
$\boldsymbol{\mathcal G}$ is multigraded admissible. Restricting the
strong identity for the $Y_{ij}$ to the multidegrees
$(A_1n_1,\ldots,A_sn_s)$ gives, after regrading,
\[
\mathcal R(\boldsymbol{\mathcal G})_{\mathbf n}
=
\sum_{i=1}^{s}\sum_{j=1}^{d_i}
(y_{ij}t_i)
\mathcal R(\boldsymbol{\mathcal G})_{\mathbf n-\mathbf e_i}
\qquad(\mathbf n\gg\mathbf0).
\]
Thus the $y_{ij}$ form a degree-one strong homogeneous joint reduction
of $\boldsymbol{\mathcal G}$. If $Q'=(y_{ij})$, then
Proposition~\ref{prop:degree-one-rees} gives
\begin{equation}
\label{eq:Qprime-degree-one}
e(Q';R)
=
e\!\left(\mathcal G^{(1)[d_1]},\ldots,
\mathcal G^{(s)[d_s]}\right).
\end{equation}

We next compute the mixed multiplicity of the Veronese collection. By
\cite[Theorem~1.6]{Cutkosky2026}, $P_{\boldsymbol{\mathcal G}}(n_1,\ldots,n_s)
=
P_{\boldsymbol{\mathcal I}}(A_1n_1,\ldots,A_sn_s).$
Comparing the coefficient of $n_1^{d_1}\cdots n_s^{d_s}$ gives
\begin{equation}
\label{eq:veronese-scaling}
e\!\left(\mathcal G^{(1)[d_1]},\ldots,
\mathcal G^{(s)[d_s]}\right)
=
\left(\prod_{i=1}^{s}A_i^{d_i}\right)
e\!\left(\mathcal I^{(1)[d_1]},\ldots,
\mathcal I^{(s)[d_s]}\right).
\end{equation}
This is also the integer coordinatewise case of the scaling principle
\cite[Lemma~3.3]{Cutkosky2026}. Combining
\eqref{eq:Qprime-degree-one} and \eqref{eq:veronese-scaling},
\begin{equation}
\label{eq:Qprime-weighted}
e(Q';R)
=
\left(\prod_iA_i^{d_i}\right)
e\!\left(\mathcal I^{(1)[d_1]},\ldots,
\mathcal I^{(s)[d_s]}\right).
\end{equation}

Since $Q'$ is $\mathfrak m$-primary and
$\sqrt{Q'}=\sqrt Q$, the ideal $Q$ is also $\mathfrak m$-primary. It is
generated by exactly $d$ elements, hence its generators form a system of
parameters. By \cite[Proposition~11.2.9(2)]{SwansonHuneke},
\[
e(Q';R)
=
\left(\prod_{i=1}^{s}\prod_{j=1}^{d_i}c_{ij}\right)e(Q;R).
\]
Because $c_{ij}=A_i/a_{ij}$, then $\prod_{i,j}c_{ij}
=
\frac{\prod_iA_i^{d_i}}{\prod_{i,j}a_{ij}}.$
Substituting this into \eqref{eq:Qprime-weighted} and cancelling the
nonzero factor $\prod_iA_i^{d_i}$ yields
\[
e(Q;R)
=
\left(\prod_{i,j}a_{ij}\right)
e\!\left(\mathcal I^{(1)[d_1]},\ldots,
\mathcal I^{(s)[d_s]}\right),
\]
as desired.
\end{proof}

\begin{rem}
Note that Proposition~\ref{prop:degree-one-rees} is recovered from
Theorem~\ref{thm:weighted-strong-rees} by taking
$a_{ij}=1$ for every $i,j$. Thus the weighted theorem extends the
degree-one formula by allowing the homogeneous generators of the strong
joint reduction to occur in arbitrary positive degrees.
\end{rem}

\begin{cor}
\label{cor:degree-one-realization}
Assume that $R/\mathfrak m$ is infinite and that
$\boldsymbol{\mathcal I}$ is multigraded admissible. Let
$\mathbf d=(d_1,\ldots,d_s)\in\mathbb N^s$ satisfy
$|\mathbf d|=d$. Then there exist elements
$x_{ij}\in I_1^{(i)}$ forming a degree-one strong homogeneous joint
reduction of type $\mathbf d$. If $Q=(x_{ij})$, then
\[
e(Q;R)
=
e\!\left(
\mathcal I^{(1)[d_1]},\ldots,
\mathcal I^{(s)[d_s]}
\right).
\]
\end{cor}

\begin{proof}
By Theorem~\ref{thm:strong-existence}, there exist
$x_{ij}\in J_i\subseteq I_1^{(i)}$ forming a degree-one strong
homogeneous joint reduction. The conclusion follows from
Theorem~\ref{thm:weighted-strong-rees} by taking $a_{ij}=1$ for all
$i,j$.
\end{proof}

\begin{cor}
\label{cor:strong-realization}
Assume that $R/\mathfrak m$ is infinite and that
$\boldsymbol{\mathcal I}$ is multigraded admissible. For every
$\mathbf d\in\N^s$ with $|\mathbf d|=d$, there is a degree-one strong
homogeneous joint reduction $\{x_{ij}\}$ such that
\[
e((x_{ij});R)
=
e\!\left(\mathcal I^{(1)[d_1]},\ldots,
\mathcal I^{(s)[d_s]}\right).
\]
\end{cor}

\begin{proof}
By Theorem~\ref{thm:strong-existence}, there exists a degree-one strong
homogeneous joint reduction $\{x_{ij}\}$ of type $\mathbf d$.
Proposition~\ref{prop:degree-one-rees} then gives the asserted
multiplicity formula.
\end{proof}

In the adic case, the weighted strong reduction becomes a classical
joint reduction of the ideals $J_i^{a_{ij}}$. Hence the following
specialization follows from Rees's mixed multiplicity theorem
\cite[Theorem~17.4.9]{SwansonHuneke} and the homogeneity of mixed
multiplicities
\cite[Theorem~17.4.2 and Definition~17.4.3]{SwansonHuneke}.

\begin{cor}
\label{cor:weighted-strong-rees-adic}
Let $(R,\mathfrak m)$ be a $d$-dimensional Noetherian local ring, and
let $J_1,\ldots,J_s$ be $\mathfrak m$-primary ideals. For each $i$, let
$\mathcal I^{(i)}={J_i^n}_{n\geq0}$ be the $J_i$-adic family. Let $\mathbf d=(d_1,\ldots,d_s)\in\mathbb N^s$ satisfy
$d_1+\cdots+d_s=d$, and suppose that
$\mathcal X={x_{ij}\mid 1\leq i\leq s,\ 1\leq j\leq d_i}$ is a
strong homogeneous joint reduction of type $\mathbf d$, where
$x_{ij}\in J_i^{a_{ij}}$ and $a_{ij}\geq1$. Set $Q=(x_{ij})$. Then the elements $x_{ij}$ form a classical joint reduction of the
$d$ ideals $
J_1^{a_{11}},\ldots,J_1^{a_{1d_1}},
\ldots,
J_s^{a_{s1}},\ldots,J_s^{a_{sd_s}}.
$ Hence,
\[
e(Q;R)
=
\left(\prod_{i=1}^{s}\prod_{j=1}^{d_i}a_{ij}\right)
e\!\left(J_1^{[d_1]},\ldots,J_s^{[d_s]}\right).
\]

In addition, one has
$$
e\!\left(
\mathcal I^{(1)[d_1]},\ldots,\mathcal I^{(s)[d_s]}
\right)
=
e\!\left(J_1^{[d_1]},\ldots,J_s^{[d_s]}\right).
$$
\end{cor}

\begin{proof}
Since $\mathcal X$ is a strong homogeneous joint reduction, there exists
$\mathbf n_0\in\mathbb N^s$ such that

$$
J_1^{n_1}\cdots J_s^{n_s}
=
\sum_{i=1}^{s}\sum_{j=1}^{d_i}
x_{ij}
J_1^{n_1}\cdots J_i^{n_i-a_{ij}}\cdots J_s^{n_s}
$$

for every $\mathbf n\geq\mathbf n_0$. Now, for each $i$, set $b_i=\sum_{j=1}^{d_i}a_{ij}$, and choose $r\gg0$ such
that $((r+1)b_1,\ldots,(r+1)b_s)\geq\mathbf n_0$. Evaluating the strong
identity at $n_i=(r+1)b_i$ gives

$$
\prod_{i=1}^{s}J_i^{(r+1)b_i}
=
\sum_{i=1}^{s}\sum_{j=1}^{d_i}
x_{ij}
J_i^{(r+1)b_i-a_{ij}}
\prod_{h\neq i}J_h^{(r+1)b_h}.
$$

Set $K_{ij}=J_i^{a_{ij}}$ and $P=\prod_{i,j}K_{ij}$. Since
$P=\prod_iJ_i^{b_i}$, we have
$P^{r+1}=\prod_iJ_i^{(r+1)b_i}$. Moreover, for every pair $(i,j)$,

$$
K_{ij}^{r}
\prod_{(h,\ell)\neq(i,j)}K_{h\ell}^{r+1}
=
J_i^{(r+1)b_i-a_{ij}}
\prod_{h\neq i}J_h^{(r+1)b_h}.
$$

Hence the preceding identity may be rewritten as

$$
P^{r+1}
=
\left(
\sum_{i,j}
x_{ij}\prod_{(h,\ell)\neq(i,j)}K_{h\ell}
\right)P^r.
$$

Therefore ${x_{ij}}$ is a classical joint reduction of the $d$-tuple
${K_{ij}}$. By Rees's mixed multiplicity theorem
\cite[Theorem~17.4.9]{SwansonHuneke}, one obtains
$$
e(Q;R)
=
e\!\left(
K_{11},\ldots,K_{1d_1},
\ldots,
K_{s1},\ldots,K_{sd_s}
\right).
$$

Since $K_{ij}=J_i^{a_{ij}}$, the homogeneity of classical mixed
multiplicities
\cite[Theorem~17.4.2 and Definition~17.4.3]{SwansonHuneke} yields

$$
e(Q;R)
=
\left(\prod_{i=1}^{s}\prod_{j=1}^{d_i}a_{ij}\right)
e\!\left(J_1^{[d_1]},\ldots,J_s^{[d_s]}\right).
$$

Hence, by \cite[Equation~(9)]{Cutkosky2026},
$$
e\!\left(
\mathcal I^{(1)[d_1]},\ldots,\mathcal I^{(s)[d_s]}
\right)
=
\lim_{n\to\infty}
\frac{
e((J_1^n)^{[d_1]},\ldots,(J_s^n)^{[d_s]})
}{n^d}.
$$

By the homogeneity of classical mixed multiplicities, we get

$$
e((J_1^n)^{[d_1]},\ldots,(J_s^n)^{[d_s]})
=
n^d e(J_1^{[d_1]},\ldots,J_s^{[d_s]}),
$$

and therefore
$e\!\left(
\mathcal I^{(1)[d_1]},\ldots,\mathcal I^{(s)[d_s]}
\right)
=
e(J_1^{[d_1]},\ldots,J_s^{[d_s]}).
$
\end{proof}

In particular, the degree-one adic specialization of
Theorem~\ref{thm:weighted-strong-rees} coincides with Rees's fundamental
theorem on mixed multiplicities
\cite[Theorem~17.4.9]{SwansonHuneke}.

\begin{cor}
\label{cor:strong-rees-classical}
Let $(R,\mathfrak m)$ be a $d$-dimensional Noetherian local ring with
infinite residue field, and let $J_1,\ldots,J_s$ be
$\mathfrak m$-primary ideals. Let
$\mathbf d=(d_1,\ldots,d_s)\in\mathbb N^s$ satisfy
$d_1+\cdots+d_s=d$. Then there exist elements $x_{ij}\in J_i$, for
$1\leq i\leq s$ and $1\leq j\leq d_i$, forming a degree-one strong
homogeneous joint reduction of the adic families
$\mathcal I^{(i)}=\{J_i^n\}_{n\geq0}$. If $Q=(x_{ij})$, then
\[
e(Q;R)
=
e(J_1^{[d_1]},\ldots,J_s^{[d_s]};R).
\]

\end{cor}

\begin{proof}
Since the residue field is infinite and the ideals $J_i$ are
$\mathfrak m$-primary, the radical conditions in
\cite[Theorem~17.3.1]{SwansonHuneke} are automatically satisfied for
the list in which $J_i$ occurs $d_i$ times. Hence there exist
$x_{ij}\in J_i$ forming a classical joint reduction of this list. Note that  the classical joint-reduction identity implies the
eventual multigraded identity
\[
J_1^{n_1}\cdots J_s^{n_s}
=
\sum_{i=1}^s\sum_{j=1}^{d_i}
x_{ij}
J_1^{n_1}\cdots J_i^{n_i-1}\cdots J_s^{n_s}
\qquad(\mathbf n\gg\mathbf0).
\]
Thus $\{x_{ij}\}$ is a degree-one strong homogeneous joint reduction.
The stated multiplicity formula now follows from
Theorem~\ref{thm:weighted-strong-rees}, or directly from
\cite[Theorem~17.4.9]{SwansonHuneke}.
\end{proof}

We now specialize the weighted strong theorem to Noetherian graded
families. After passing to a sufficiently divisible common Veronese,
the families become adic, so the result reduces to the preceding
classical specialization.

\begin{cor}
\label{cor:weighted-strong-rees-noetherian}
Let $(R,\mathfrak m)$ be a $d$-dimensional Noetherian local ring and let
$\boldsymbol{\mathcal I}
=(\mathcal I^{(1)},\ldots,\mathcal I^{(s)})$
be a multigraded admissible collection of graded
$\mathfrak m$-primary families. Assume that each
$\mathcal I^{(i)}$ is Noetherian.
Let $\mathbf d=(d_1,\ldots,d_s)\in\mathbb N^s$ satisfy
$d_1+\cdots+d_s=d$, and let $\{x_{ij}\}$ be a strong homogeneous joint
reduction of type $\mathbf d$, with
$x_{ij}\in I_{a_{ij}}^{(i)}$ and $a_{ij}\geq1$. Set $Q=(x_{ij})$.
Then there exists a sufficiently divisible integer $c\geq1$, divisible
by every $a_{ij}$, such that
\[
I_{cm}^{(i)}=\bigl(I_c^{(i)}\bigr)^m
\qquad
\text{for every }i\text{ and }m\geq0.
\]
Moreover, after setting $y_{ij}=x_{ij}^{c/a_{ij}}$, the elements
$\{y_{ij}\}$ form a classical joint reduction of the list in which
$I_c^{(i)}$ occurs $d_i$ times. Consequently,
\[
e(Q;R)
=
\left(\prod_{i,j}a_{ij}\right)
e\!\left(
\mathcal I^{(1)[d_1]},\ldots,
\mathcal I^{(s)[d_s]}
\right).
\]
\end{cor}

\begin{proof}
For each $i$, the Noetherianity of $\mathcal I^{(i)}$ means that its
Rees algebra
$\mathcal R(\mathcal I^{(i)})
=\bigoplus_{n\geq0}I_n^{(i)}t^n$
is a finitely generated positively graded $R$-algebra. By the standard
Veronese criterion
\cite[Theorem~2.1]{HerzogHibiTrung}, there exists $c_i\geq1$ such that
the $c_i$-th Veronese subalgebra of
$\mathcal R(\mathcal I^{(i)})$ is standard graded.

Choose a common multiple $c$ of $c_1,\ldots,c_s$ which is also divisible
by every $a_{ij}$. Since every sufficiently divisible Veronese of a
standard graded algebra is again standard graded, we may choose $c$ so
that $
I_{cm}^{(i)}=\bigl(I_c^{(i)}\bigr)^m$
for every $i$ and $m\geq0.$ Set $b_{ij}=c/a_{ij}$ and $y_{ij}=x_{ij}^{b_{ij}}$. By
Lemma~\ref{lem:power-stability}, the elements $y_{ij}$ form a strong
homogeneous joint reduction of $\boldsymbol{\mathcal I}$, now with
homogeneous degree $c$ in every direction.

Consider the $c$-th coordinatewise Veronese collection. After regrading,
its $i$-th component is
${I_{cm}^{(i)}}*{m\geq0}
={(I_c^{(i)})^m}*{m\geq0}$.
Thus the Veronese collection consists of the $I_c^{(i)}$-adic families,
and the elements $y_{ij}$ have homogeneous degree one after regrading.
By Corollary~\ref{cor:weighted-strong-rees-adic}, the elements
${y_{ij}}$ form a classical joint reduction of the list in which
$I_c^{(i)}$ occurs $d_i$ times. Therefore Rees's mixed multiplicity
theorem \cite[Theorem~17.4.9]{SwansonHuneke} gives

$$
e(Q^{(c)};R)
=
e\!\left(
I_c^{(1)[d_1]},\ldots,I_c^{(s)[d_s]}
\right).
$$

Since $y_{ij}=x_{ij}^{b_{ij}}$ for every $i,j$, we have
$\sqrt{Q^{(c)}}=\sqrt Q$. Thus $Q$ is $\mathfrak m$-primary because
$Q^{(c)}$ is $\mathfrak m$-primary. Moreover, $Q$ is generated by
$d_1+\cdots+d_s=d$ elements, so the $x_{ij}$ form a system of
parameters. Hence
\cite[Proposition~11.2.9(2)]{SwansonHuneke} yields

$$
e(Q^{(c)};R)
=
\left(\prod_{i,j}b_{ij}\right)e(Q;R)
=
\left(\prod_{i,j}\frac{c}{a_{ij}}\right)e(Q;R).
$$

Since there are precisely $d$ indices $(i,j)$,
$\prod_{i,j}(c/a_{ij})=c^d/\prod_{i,j}a_{ij}$. Combining this with the
preceding formula for $e(Q^{(c)};R)$ gives

$$
e(Q;R)
=
\frac{\prod_{i,j}a_{ij}}{c^d}
e\!\left(
I_c^{(1)[d_1]},\ldots,I_c^{(s)[d_s]}
\right).
$$

It remains to identify the mixed multiplicity of the original families.
By \cite[Equation~(9)]{Cutkosky2026},

$$
e\!\left(
\mathcal I^{(1)[d_1]},\ldots,
\mathcal I^{(s)[d_s]}
\right)
=
\lim_{n\to\infty}
\frac{
e(I_n^{(1)[d_1]},\ldots,I_n^{(s)[d_s]})
}{n^d}.
$$

Since this limit exists, it is unchanged after restriction to the
subsequence $n=cm$. Hence

$$
e\!\left(
\mathcal I^{(1)[d_1]},\ldots,
\mathcal I^{(s)[d_s]}
\right)
=
\lim_{m\to\infty}
\frac{
e(I_{cm}^{(1)[d_1]},\ldots,I_{cm}^{(s)[d_s]})
}{(cm)^d}.
$$

Using $I_{cm}^{(i)}=(I_c^{(i)})^m$ and the homogeneity of classical
mixed multiplicities
\cite[Theorem~17.4.2 and Definition~17.4.3]{SwansonHuneke}, we obtain

\[
e(I_{cm}^{(1)[d_1]},\ldots,I_{cm}^{(s)[d_s]})
=
e\!\left(((I_c^{(1)})^m)^{[d_1]},\ldots,((I_c^{(s)})^m)^{[d_s]}\right)
=
m^d e\!\left(I_c^{(1)[d_1]},\ldots,I_c^{(s)[d_s]}\right).
\]

Therefore
$
e\!\left(
\mathcal I^{(1)[d_1]},\ldots,
\mathcal I^{(s)[d_s]}
\right)
=
\frac{1}{c^d}
e\!\left(
I_c^{(1)[d_1]},\ldots,I_c^{(s)[d_s]}
\right).
$
Substitution into the previously obtained formula for $e(Q;R)$ gives

$$
e(Q;R)
=
\left(\prod_{i,j}a_{ij}\right)
e\!\left(
\mathcal I^{(1)[d_1]},\ldots,
\mathcal I^{(s)[d_s]}
\right),
$$

as required.
\end{proof}

\subsection*{Weak versus Strong Joint Reductions}

The weak and strong notions reflect different kinds of asymptotic
control: the former is levelwise, while the latter is governed by a
fixed homogeneous system in the multigraded Rees algebra. In general,
these notions are not equivalent, as the next proposition shows.
\begin{prop}
\label{prop:weak-not-strong}
There exists a graded $\mathfrak m$-primary family on a one-dimensional
regular local ring which admits a weak asymptotic joint reduction of type
$1$ but no strong homogeneous joint reduction of type $1$.
\end{prop}

\begin{proof}
Let $R=k[[x]]$, $\mathfrak m=(x)$, with $k$ infinite, and fix an
irrational real number $\alpha>0$. Set $I_0=R$ and
$I_n=(x^{\lceil\alpha n\rceil})$ for $n\geq1$. Since
$\lceil\alpha m\rceil+\lceil\alpha n\rceil
\geq\lceil\alpha(m+n)\rceil$, this is a graded
$\mathfrak m$-primary family. Each $I_n$ is principal, so its generator
$x^{\lceil\alpha n\rceil}$ is a reduction of $I_n$; hence the family
admits a weak asymptotic joint reduction of type $1$.

Suppose a strong homogeneous joint reduction of type $1$ existed. Then
there would be $a>0$ and $y\in I_a$ such that
$I_n=yI_{n-a}$ for all $n\gg0$. Writing $y=ux^b$ with $u$ a unit gives
\[
\lceil\alpha n\rceil-
\lceil\alpha(n-a)\rceil=b
\qquad(n\gg0).
\]
Since $\alpha$ is irrational, the fractional parts of $\alpha n$ are
dense in $[0,1]$, and the difference on the left assumes both
$\lfloor\alpha a\rfloor$ and $\lceil\alpha a\rceil$ infinitely often.
This contradiction proves the assertion.
\end{proof}


\begin{rem}\label{rem:adic-case}
Suppose that
$\mathcal I^{(i)}=\{J_i^n\}_{n\geq0}$ for every $i$. Then a weak
asymptotic joint reduction is simply a choice, at each sufficiently
large level $n$, of a classical joint reduction of the repeated ideals
$J_1^n,\ldots,J_s^n$. If, moreover, a strong homogeneous joint reduction has
$a_{ij}=1$ for all $i,j$, its defining identity becomes
\[
J_1^{n_1}\cdots J_s^{n_s}
=
\sum_{i=1}^{s}\sum_{j=1}^{r_i}
x_{ij}
J_1^{n_1}\cdots J_i^{n_i-1}\cdots J_s^{n_s}
\qquad(\mathbf n\gg\mathbf0),
\]
which is the classical multigraded joint-reduction identity; compare
\cite[Definition~1.6]{SarkarVerma}. Then both notions specialize naturally to the classical theory in the
adic case. Their corresponding multiplicity formulas are given in
Corollaries~\ref{cor:weak-rees-adic} and
\ref{cor:weighted-strong-rees-adic}, with the degree-one strong case
recovering Rees's fundamental mixed-multiplicity theorem
\cite[Theorem~17.4.9]{SwansonHuneke}.
\end{rem}

\section{A converse to the Rees theorem for graded families}
\label{sec:converse}

The classical Rees theorem identifies the mixed multiplicity associated
to a joint reduction with the Hilbert--Samuel multiplicity of the ideal
generated by it. Its converse is substantially more delicate. In the
classical setting, \cite[Theorem~17.6.1]{SwansonHuneke} shows that over a
formally equidimensional local ring, equality of the relevant local
multiplicities, together with the height and radical hypotheses, forces
the given elements to form a joint reduction. For graded families the
weak/strong distinction becomes decisive. First, the next example shows that  the converse to the weak asymptotic Rees theorem fails in general.

\begin{ex}
\label{ex:weak-converse}
Let $R=k[[x]]$, $\mathfrak m=(x)$, and let
$\mathcal I=\{I_n\}_{n\geq0}$ be the adic family $I_n=(x^n)$. Then
$e(\mathcal I)=1$. For every $n\geq1$, choose
$x_n=x^{n+1}\in I_n$ and set $Q_n=(x_n)$. Since
$e(Q_n;R)=n+1$,
\[
\lim_{n\to\infty}\frac{e(Q_n;R)}{n}
=1=e(\mathcal I).
\]
Nevertheless $Q_n$ is not a reduction of $I_n$ for any $n$: if
$Q_nI_n^r=I_n^{r+1}$ for some $r\geq0$, then
$(x^{nr+n+1})=(x^{nr+n})$, which is impossible. Hence the numerical
identity in Theorem~\ref{thm:weak-rees} does not characterize weak
asymptotic joint reductions.
\end{ex}


We now prove a converse in a form that extends the full localized
converse of Rees's mixed-multiplicity theorem from ideals to
multigraded admissible families.

\begin{thm}[Localized converse of the strong Rees theorem]
\label{thm:localized-strong-converse}
Let $(R,\mathfrak m)$ be a formally equidimensional Noetherian local
ring, and let
$\boldsymbol{\mathcal I}
=(\mathcal I^{(1)},\ldots,\mathcal I^{(s)})$
be a multigraded admissible collection of graded families of ideals.
Assume that there exists an ideal $\mathfrak a$ such that
$\sqrt{I_n^{(i)}}=\mathfrak a$ for every $i=1,\ldots,s$ and every
$n\geq1$. Set $k=\operatorname{ht}\mathfrak a$, and let
$\mathbf d=(d_1,\ldots,d_s)\in\mathbb N^s$ satisfy $|\mathbf d|=k$. Choose positive integers $a_{ij}$ and elements
$x_{ij}\in I_{a_{ij}}^{(i)}$ for
$1\leq i\leq s$ and $1\leq j\leq d_i$, and set $Q=(x_{ij})$.
Assume that $\sqrt Q=\mathfrak a$ and that, for every
$P\in\operatorname{Min}(Q)$,
\[
e(QR_P;R_P)
=
\left(\prod_{i=1}^{s}\prod_{j=1}^{d_i}a_{ij}\right)
e\!\left(
\mathcal I_P^{(1)[d_1]},\ldots,
\mathcal I_P^{(s)[d_s]};R_P
\right),
\]
where $\mathcal I_P^{(i)}=\{I_n^{(i)}R_P\}_{n\geq0}$.
Then $\{x_{ij}\}$ is a strong homogeneous joint reduction of
$\boldsymbol{\mathcal I}$ of type $\mathbf d$, with homogeneous degrees
$a_{ij}$.
\end{thm}

\begin{proof}
For every $i$ with $d_i>0$, set
$A_i=\operatorname{lcm}(a_{i1},\ldots,a_{id_i})$ and
$c_{ij}=A_i/a_{ij}$; if $d_i=0$, put $A_i=1$. Define
$y_{ij}=x_{ij}^{c_{ij}}$ and $Q'=(y_{ij})$. Since each $y_{ij}$ is a
positive power of $x_{ij}$, we have
$\sqrt{Q'}=\sqrt Q=\mathfrak a$. Fix $P\in\operatorname{Min}(Q)$. Then
$\dim R_P=\operatorname{ht}P=k$, and $QR_P$ is $PR_P$-primary and
generated by $k$ elements. Hence these generators form a system of
parameters of $R_P$. By
\cite[Proposition~11.2.9(2)]{SwansonHuneke},
\[
e(Q'R_P;R_P)
=
\left(\prod_{i,j}c_{ij}\right)e(QR_P;R_P).
\]
Using the assumed localized equality and
$c_{ij}a_{ij}=A_i$, we obtain
\[
e(Q'R_P;R_P)
=
\left(\prod_{i=1}^{s}A_i^{d_i}\right)
e\!\left(
\mathcal I_P^{(1)[d_1]},\ldots,
\mathcal I_P^{(s)[d_s]};R_P
\right).
\]

For each $i$, let
$\mathcal G^{(i)}=\{I_{A_i n}^{(i)}\}_{n\geq0}$.
Localization commutes with passage to these coordinatewise Veronese
families. The Veronese scaling formula for mixed multiplicities gives
\[
e\!\left(
\mathcal G_P^{(1)[d_1]},\ldots,
\mathcal G_P^{(s)[d_s]};R_P
\right)
=
\left(\prod_{i=1}^{s}A_i^{d_i}\right)
e\!\left(
\mathcal I_P^{(1)[d_1]},\ldots,
\mathcal I_P^{(s)[d_s]};R_P
\right).
\]
Consequently,
\begin{equation}
\label{eq:localized-converse-G}
e(Q'R_P;R_P)
=
e\!\left(
\mathcal G_P^{(1)[d_1]},\ldots,
\mathcal G_P^{(s)[d_s]};R_P
\right).
\end{equation}

Set $K_i=I_{A_i}^{(i)}$. We claim that
$\boldsymbol{\mathcal G}$ is admissible with respect to
$\mathbf K=(K_1,\ldots,K_s)$. Indeed, if the original collection is
admissible with respect to $\mathbf J=(J_1,\ldots,J_s)$, then its
coordinatewise Veronese Rees algebra is finite over the corresponding
Veronese of $\mathcal R(\mathbf J)$. After regrading, the latter is
$\mathcal R(J_1^{A_1},\ldots,J_s^{A_s})$. Since
$J_i^{A_i}\subseteq K_i$ and
$K_i^n\subseteq I_{A_i n}^{(i)}$, one has
\[
\mathcal R(J_1^{A_1},\ldots,J_s^{A_s})
\subseteq
\mathcal R(\mathbf K)
\subseteq
\mathcal R(\boldsymbol{\mathcal G}).
\]
Thus the same finite set of module generators over the first algebra
also generates $\mathcal R(\boldsymbol{\mathcal G})$ over
$\mathcal R(\mathbf K)$. Recall that localization preserves this finiteness. Moreover,
$\sqrt{K_iR_P}=PR_P$, so every $K_iR_P$ is $PR_P$-primary.
By the invariance of mixed multiplicities for admissible collections,
\eqref{eq:localized-converse-G} becomes
\begin{equation}
\label{eq:localized-converse-classical}
e(Q'R_P;R_P)
=
e\!\left(
(K_1R_P)^{[d_1]},\ldots,
(K_sR_P)^{[d_s]};R_P
\right)
\end{equation}
for every $P\in\operatorname{Min}(Q')$. Since $y_{ij}\in K_i$, while $Q'$ and all the $K_i$ have the same
radical $\mathfrak a$ and the same height $k$, we may apply
\cite[Theorem~17.6.1]{SwansonHuneke} to the list in which $K_i$ occurs
$d_i$ times. The localized equalities required there are precisely
\eqref{eq:localized-converse-classical}. Hence the $y_{ij}$ form a
classical joint reduction of
\[
\underbrace{K_1,\ldots,K_1}_{d_1},
\ldots,
\underbrace{K_s,\ldots,K_s}_{d_s}.
\]

Let $S=\{i\mid d_i>0\}$ and set $P_0=\prod_{i\in S}K_i^{d_i},$ and  $H=
\sum_{i\in S}\sum_{j=1}^{d_i}
y_{ij}K_i^{d_i-1}
\prod_{\substack{h\in S\\h\neq i}}K_h^{d_h}.$
The classical joint-reduction condition says that $H$ is a reduction
of $P_0$. Hence $HP_0^r=P_0^{r+1}$ for some $r\geq0$. Expanding this
identity and multiplying by arbitrary additional powers of the active
$K_i$, as well as arbitrary powers of the inactive $K_h$, gives
\begin{equation}
\label{eq:B-tail-localized-converse}
K_1^{n_1}\cdots K_s^{n_s}
=
\sum_{i=1}^{s}\sum_{j=1}^{d_i}
y_{ij}
K_1^{n_1}\cdots K_i^{n_i-1}\cdots K_s^{n_s}
\end{equation}
for every $\mathbf n\gg\mathbf0$. Let $\mathcal B=R[K_1t_1,\ldots,K_st_s]$ and
$\mathcal C=\mathcal R(\boldsymbol{\mathcal G})$.
Since $\boldsymbol{\mathcal G}$ is $\mathbf K$-admissible,
$\mathcal C$ is a finite graded $\mathcal B$-module. Choose homogeneous
generators $u_1,\ldots,u_v$ with degrees
$\mathbf b_1,\ldots,\mathbf b_v$. For $\mathbf n\gg\mathbf0$, $\mathcal C_{\mathbf n}
=
\sum_{\ell=1}^{v}
\mathcal B_{\mathbf n-\mathbf b_\ell}u_\ell.$
Taking $\mathbf n$ sufficiently large that every
$\mathbf n-\mathbf b_\ell$ lies in the range of
\eqref{eq:B-tail-localized-converse}, and translating that identity into
homogeneous components, yields $\mathcal C_{\mathbf n}
=
\sum_{i,j}(y_{ij}t_i)\mathcal C_{\mathbf n-\mathbf e_i}.$
Thus $\{y_{ij}\}$ is a degree-one strong homogeneous joint reduction of
the Veronese collection $\boldsymbol{\mathcal G}$.

It remains to return to the original grading. Let
$\mathcal A=\mathcal R(\boldsymbol{\mathcal I})$ and let
$\mathcal D\subseteq\mathcal A$ be its coordinatewise
$\mathbf A$-Veronese subalgebra, so that after regrading
$\mathcal D=\mathcal R(\boldsymbol{\mathcal G})$. The algebra
$\mathcal A$ is finite over $\mathcal D$. Indeed, if
$\mathcal B_0=\mathcal R(\mathbf J)$ is an admissible standard
multigraded base, then $\mathcal A$ is finite over $\mathcal B_0$, while
$\mathcal B_0$ is finite over its coordinatewise
$\mathbf A$-Veronese subalgebra. Hence $\mathcal A$ is finite over
$\mathcal B_0^{(\mathbf A)}$, and since
$\mathcal B_0^{(\mathbf A)}\subseteq\mathcal D\subseteq\mathcal A$,
the same finite set of generators makes $\mathcal A$ a finite
$\mathcal D$-module. Choose homogeneous $\mathcal D$-module generators
$v_1,\ldots,v_q$. In the original grading put
$Y_{ij}=y_{ij}t_i^{A_i}$. The strong identity in $\mathcal D$, together
with the finite $\mathcal D$-module generation of $\mathcal A$, gives $\mathcal A_{\boldsymbol\nu}
=
\sum_{i,j}
Y_{ij}\mathcal A_{\boldsymbol\nu-A_i\mathbf e_i}$  for $\boldsymbol\nu\gg\mathbf0$.
Then $\{y_{ij}\}$ is a strong homogeneous joint reduction of the
original collection, with degrees $A_i$.

Now, put $X_{ij}=x_{ij}t_i^{a_{ij}}$. Since
$Y_{ij}=X_{ij}^{c_{ij}}$, one has
$(Y_{ij})\mathcal A\subseteq(X_{ij})\mathcal A$. Hence the vanishing of
the sufficiently large multigraded components of
$\mathcal A/(Y_{ij})\mathcal A$ implies the corresponding vanishing for
$\mathcal A/(X_{ij})\mathcal A$. By
Proposition~\ref{prop:intrinsic}, the original elements $x_{ij}$ form a
strong homogeneous joint reduction of type $\mathbf d$ and degrees
$a_{ij}$.
\end{proof}

The maximal-height $\mathfrak m$-primary converse used earlier in the
strong theory is now an immediate specialization of the localized
theorem.

\begin{cor}
\label{cor:strong-converse}
\label{thm:strong-converse}
Let $(R,\mathfrak m)$ be a formally equidimensional Noetherian local ring
of dimension $d\geq1$, and let
$\boldsymbol{\mathcal I}
=(\mathcal I^{(1)},\ldots,\mathcal I^{(s)})$
be a multigraded admissible collection of graded
$\mathfrak m$-primary families. Let
$\mathbf d=(d_1,\ldots,d_s)\in\mathbb N^s$ satisfy $|\mathbf d|=d$.
Choose positive integers $a_{ij}$ and elements
$x_{ij}\in I_{a_{ij}}^{(i)}$, and set $Q=(x_{ij})$. Assume that $Q$ is
$\mathfrak m$-primary and
\[
e(Q;R)
=
\left(\prod_{i=1}^{s}\prod_{j=1}^{d_i}a_{ij}\right)
e\!\left(
\mathcal I^{(1)[d_1]},\ldots,
\mathcal I^{(s)[d_s]}
\right).
\]
Then $\{x_{ij}\}$ is a strong homogeneous joint reduction of type
$\mathbf d$ with homogeneous degrees $a_{ij}$.
\end{cor}

\begin{proof}
Apply Theorem~\ref{thm:localized-strong-converse} with
$\mathfrak a=\mathfrak m$. Since every $I_n^{(i)}$ and $Q$ are
$\mathfrak m$-primary, their radicals are $\mathfrak m$, one has
$\operatorname{ht}\mathfrak m=d$, and
$\operatorname{Min}(Q)=\{\mathfrak m\}$. Thus the localized condition in
Theorem~\ref{thm:localized-strong-converse} reduces to the single
equality at $P=\mathfrak m$, which is precisely the hypothesis above.
\end{proof}

Combining the preceding converse with
Theorem~\ref{thm:weighted-strong-rees} gives the numerical
characterization in the original $\mathfrak m$-primary setting.

\begin{cor}
\label{cor:numerical-characterization}
Let $(R,\mathfrak m)$ be a formally equidimensional Noetherian local ring
of dimension $d\geq1$, and let
$\boldsymbol{\mathcal I}
=(\mathcal I^{(1)},\ldots,\mathcal I^{(s)})$
be a multigraded admissible collection of graded
$\mathfrak m$-primary families. Let
$\mathbf d=(d_1,\ldots,d_s)\in\mathbb N^s$ satisfy $|\mathbf d|=d$.
Choose positive integers $a_{ij}$ and elements
$x_{ij}\in I_{a_{ij}}^{(i)}$, and assume that
$Q=(x_{ij})$ is $\mathfrak m$-primary. Then the following conditions are
equivalent:
\begin{enumerate}
\item $\{x_{ij}\}$ is a strong homogeneous joint reduction of type
$\mathbf d$ with homogeneous degrees $a_{ij}$;
\item
\[
e(Q;R)
=
\left(\prod_{i=1}^{s}\prod_{j=1}^{d_i}a_{ij}\right)
e\!\left(
\mathcal I^{(1)[d_1]},\ldots,
\mathcal I^{(s)[d_s]}
\right).
\]
\end{enumerate}
\end{cor}

\begin{proof}
The implication $(1)\Rightarrow(2)$ is
Theorem~\ref{thm:weighted-strong-rees}, while
$(2)\Rightarrow(1)$ is
Corollary~\ref{cor:strong-converse}.
\end{proof}

The next corollary shows that Theorem~\ref{thm:localized-strong-converse}
recovers the full localized converse of Swanson--Huneke
\cite[Theorem~17.6.1]{SwansonHuneke}.

\begin{cor}
\label{cor:swanson-full-specialization}
Let $(R,\mathfrak m)$ be a formally equidimensional Noetherian local
ring, let $I_1,\ldots,I_k$ be ideals of $R$, and choose
$x_i\in I_i$ for $i=1,\ldots,k$. Assume that
$(x_1,\ldots,x_k)$ and the ideals $I_i$ have the same radical and the
same height $k$. If
\[
e((x_1,\ldots,x_k)R_P;R_P)
=
e(I_1R_P,\ldots,I_kR_P;R_P)
\]
for every $P\in\operatorname{Min}(x_1,\ldots,x_k)$, then
$(x_1,\ldots,x_k)$ is a joint reduction of
$(I_1,\ldots,I_k)$.
\end{cor}

\begin{proof}
For each $i$, take the adic family
$\mathcal I^{(i)}=\{I_i^n\}_{n\geq0}$, and set
$s=k$, $d_i=1$, and $a_{i1}=1$. These adic families form a multigraded
admissible collection, and their localized mixed multiplicities are the
classical mixed multiplicities. Thus the hypotheses of
Theorem~\ref{thm:localized-strong-converse} are precisely the hypotheses
above. Hence $\{x_1,\ldots,x_k\}$ is a degree-one strong homogeneous
joint reduction. Therefore, for every $\mathbf n\gg\mathbf0$, $I_1^{n_1}\cdots I_k^{n_k}
=
\sum_{i=1}^{k}
x_iI_1^{n_1}\cdots I_i^{n_i-1}\cdots I_k^{n_k}.$
Taking $n_1=\cdots=n_k=n\gg0$ and setting
$L=I_1\cdots I_k$ and
$H=\sum_i x_i\prod_{j\neq i}I_j$, we obtain
$L^n=HL^{n-1}$. Hence $H$ is a reduction of $L$, which is exactly the
statement that $(x_1,\ldots,x_k)$ is a classical joint reduction of
$(I_1,\ldots,I_k)$.
\end{proof}

\begin{defn}
Let $\mathcal I=\{I_n\}_{n\geq0}$ and
$\mathcal J=\{J_n\}_{n\geq0}$ be graded families such that
$I_n\subseteq J_n$ for every $n\geq0$. We say that
$\mathcal I$ is a \emph{reduction} of $\mathcal J$ if
$\mathcal R(\mathcal J)$ is a finite
$\mathcal R(\mathcal I)$-module.
\end{defn}

\begin{defn}
Let $\mathcal I=\{I_n\}_{n\geq0}$ be a Noetherian graded family. Its
\emph{analytic spread} is defined by
\[
\ell(\mathcal I)
:=
\dim\!\left(
\frac{\mathcal R(\mathcal I)}
{\mathfrak m\mathcal R(\mathcal I)}
\right).
\]
\end{defn}

\begin{lem}
\label{lem:equimultiple-homogeneous-reduction}
Let $(R,\mathfrak m)$ be a Noetherian local ring and let
$\mathcal I=\{I_n\}_{n\geq0}$ be a Noetherian graded family.
Set $\mathcal A=\mathcal R(\mathcal I)$ and $k=\ell(\mathcal I).$
Then there exist positive integers $a_1,\ldots,a_k$ and elements
$x_j\in I_{a_j}$ such that $I_n
=
\sum_{j=1}^{k}x_jI_{n-a_j}$ for  $n\gg0.$
Thus $\{x_1,\ldots,x_k\}$ is a strong homogeneous reduction of
$\mathcal I$.
\end{lem}

\begin{proof}
Let
$\mathcal F(\mathcal I)
=\mathcal A/\mathfrak m\mathcal A$
be the fiber algebra. Since $\mathcal A$ is Noetherian,
$\mathcal F(\mathcal I)$ is a finitely generated graded algebra over
$R/\mathfrak m$, of dimension $k$. By homogeneous Noether
normalization, there exist homogeneous elements
$\overline X_1,\ldots,\overline X_k$ of positive degrees
$a_1,\ldots,a_k$ such that
$\mathcal F(\mathcal I)$ is finite over
$(R/\mathfrak m)[\overline X_1,\ldots,\overline X_k]$. Choose homogeneous lifts
$X_j=x_jt^{a_j}\in\mathcal A$.
It follows that, for all sufficiently large $n$, $I_n
=
\sum_{j=1}^{k}x_jI_{n-a_j}
+
\mathfrak m I_n.$
Since $I_n$ is a finite $R$-module, Nakayama's lemma gives $I_n
=
\sum_{j=1}^{k}x_jI_{n-a_j},$
as desired.
\end{proof}

Now, we are able to deduce a B\"oger-type reduction criterion for graded families,
extending \cite{Boger1969} to admissible families and localized mixed
multiplicities.
\begin{thm}
\label{thm:boger-graded-families}
Let $(R,\mathfrak m)$ be a formally equidimensional Noetherian local
ring, and let
$\mathcal I=\{I_n\}_{n\geq0}\subseteq
 \mathcal J=\{J_n\}_{n\geq0}$
be admissible graded families. Assume that there exists an ideal
$\mathfrak a$ such that $\sqrt{I_n}=\sqrt{J_n}=\mathfrak a$ for
$n\geq1,$
and set $k=\operatorname{ht}\mathfrak a$. Suppose that $\ell(\mathcal I)=k.$ If $e\!\left(\mathcal I_P^{[k]};R_P\right)
=
e\!\left(\mathcal J_P^{[k]};R_P\right)$
for every $P\in\operatorname{Min}(\mathfrak a)$, then
$\mathcal I$ is a reduction of $\mathcal J$.
\end{thm}

\begin{proof}
By Lemma~\ref{lem:equimultiple-homogeneous-reduction}, there exist
positive integers $a_1,\ldots,a_k$ and elements
$x_j\in I_{a_j}$ such that $I_n=\sum_{j=1}^{k}x_jI_{n-a_j}$ for $n\gg0$.
Set $Q=(x_1,\ldots,x_k)$. Since $x_j\in I_{a_j}\subseteq\mathfrak a$, we have
$\sqrt Q\subseteq\mathfrak a$. On the other hand, the strong
reduction identity gives $I_n\subseteq Q$ for all sufficiently large
$n$. Hence $\mathfrak a=\sqrt{I_n}\subseteq\sqrt Q,$
and therefore $\sqrt Q=\mathfrak a$.

Let $P\in\operatorname{Min}(\mathfrak a)$. Since $Q$ is generated by
$k=\operatorname{ht}\mathfrak a$ elements and $\sqrt Q=\mathfrak a$,
we have $\operatorname{ht}P=k$. Thus $\dim R_P=k$. Localizing the strong homogeneous reduction of $\mathcal I$ at $P$
and applying Theorem~\ref{thm:weighted-strong-rees}, we obtain
\[
e(QR_P;R_P)
=
\left(\prod_{j=1}^{k}a_j\right)
e\!\left(\mathcal I_P^{[k]};R_P\right).
\]
By hypothesis,
\[
e(QR_P;R_P)
=
\left(\prod_{j=1}^{k}a_j\right)
e\!\left(\mathcal J_P^{[k]};R_P\right)
\]
for every $P\in\operatorname{Min}(\mathfrak a)$.
Since $x_j\in I_{a_j}\subseteq J_{a_j}$, all the hypotheses of
Theorem~\ref{thm:localized-strong-converse} are satisfied for
$\mathcal J$. Hence $\{x_1,\ldots,x_k\}$ is a strong homogeneous
reduction of $\mathcal J$. Therefore
$\mathcal R(\mathcal J)$ is finite over $R[x_1t^{a_1},\ldots,x_kt^{a_k}].$
But $R[x_1t^{a_1},\ldots,x_kt^{a_k}]
\subseteq
\mathcal R(\mathcal I)
\subseteq
\mathcal R(\mathcal J).$
Consequently $\mathcal R(\mathcal J)$ is finite over
$\mathcal R(\mathcal I)$, and hence $\mathcal I$ is a reduction of
$\mathcal J$.
\end{proof}

The previous results allow us to deduce  the classical theorem of B\"oger
\cite{Boger1969}. 
\begin{cor}
\label{cor:boger-specialization}
Let $I\subseteq J\subseteq\sqrt I$ be ideals in a formally
equidimensional local ring, and assume that
$\ell(I)=\operatorname{ht}I$. If $e(IR_P;R_P)=e(JR_P;R_P)$
for every $P\in\operatorname{Min}(I)$, then $I$ is a reduction of
$J$.
\end{cor}

\begin{proof}
Apply Theorem~\ref{thm:boger-graded-families} to the adic families
$\mathcal I=\{I^n\}_{n\geq0}$ and
$\mathcal J=\{J^n\}_{n\geq0}$.
In this case
$\ell(\mathcal I)=\ell(I)$, $e(\mathcal I_P^{[k]};R_P)=e(IR_P;R_P)$, $e(\mathcal J_P^{[k]};R_P)=e(JR_P;R_P)$,
and reduction of the two adic families is equivalent to reduction of
$I\subseteq J$.
\end{proof}

\subsection*{Declaration of generative AI and AI-assisted technologies}

During the preparation of this work, the authors used OpenAI
ChatGPT as an auxiliary tool for discussion, proof
exploration, and improvements in organization and language. It was also
used to suggest potentially relevant references, which were subsequently
verified against the original sources. All mathematical arguments,
proofs, computations, and citations were independently checked by the
authors, who take full responsibility for the content of the paper.


\section{Acknowledgments}
The first author gratefully acknowledges the Universidade Tecnol\'ogica
Federal do Paran\'a - Campus Guarapuava for the opportunity to carry out his postdoctoral
research. He also acknowledges financial support from FAPESP, grant
2025/20830-5.



\begin{thebibliography}{99}

\bibitem{Bhattacharya1957}
P. B. Bhattacharya,
\emph{The Hilbert function of two ideals},
Math. Proc. Cambridge Philos. Soc. \textbf{53} (1957), 568--575.

\bibitem{Boger1969}
E. B\"oger,
\emph{Eine Verallgemeinerung eines Multiplizit\"atensatzes von D. Rees},
J. Algebra \textbf{12} (1969), 207--215.


\bibitem{CidRuizMontano}
Y. Cid-Ruiz and J. Monta\~no,
\emph{Mixed multiplicities of graded families of ideals},
J. Algebra \textbf{590} (2022), 394--412.

\bibitem{Cutkosky2014}
S. D. Cutkosky,
\emph{Asymptotic multiplicities of graded families of ideals and linear
series},
Adv. Math. \textbf{264} (2014), 55--113.

\bibitem{CutkoskyDivisorial}
S. D. Cutkosky,
\emph{Mixed multiplicities of divisorial filtrations},
Adv. Math. \textbf{358} (2019), 106842.

\bibitem{Cutkosky2013}
S. D. Cutkosky,
\emph{Multiplicities associated to graded families of ideals},
Algebra Number Theory \textbf{7} (2013), no.~9, 2059--2083.


\bibitem{Cutkosky2026}
S. D. Cutkosky,
\emph{Multiplicities of graded families of ideals on Noetherian local rings},
arXiv:2603.06844, 2026.

\bibitem{CutkoskySarkar2022}
S. D. Cutkosky and P. Sarkar,
\emph{Multiplicities and mixed multiplicities of arbitrary filtrations},
Res. Math. Sci. \textbf{9} (2022), Art.~14.

\bibitem{CutkoskySarkarSrinivasan}
S. D. Cutkosky, P. Sarkar and H. Srinivasan,
\emph{Mixed multiplicities of filtrations},
Trans. Amer. Math. Soc. \textbf{372} (2019), no.~9, 6183--6211.

\bibitem{CutkoskySrinivasanVerma}
S. D. Cutkosky, H. Srinivasan and J. K. Verma,
\emph{Positivity of mixed multiplicities of filtrations},
Bull. London Math. Soc. \textbf{52} (2020), 335--348.

\bibitem{HerzogHibiTrung}
J. Herzog, T. Hibi and N. V. Trung,
\emph{Symbolic powers of monomial ideals and vertex cover algebras},
Adv. Math. \textbf{210} (2007), no.~1, 304--322.

\bibitem{KatzVerma1989}
D. Katz and J. K. Verma,
\emph{Extended Rees algebras and mixed multiplicities},
Math. Z. \textbf{202} (1989), 111--128.

\bibitem{MasutiSarkarVerma}
S. K. Masuti, P. Sarkar and J. K. Verma,
\emph{Hilbert polynomials of multigraded filtrations of ideals},
J. Algebra \textbf{444} (2015), 527--566.

\bibitem{Rees1984}
D. Rees,
\emph{Generalizations of reductions and mixed multiplicities},
J. London Math. Soc. (2) \textbf{29} (1984), 397--414.

\bibitem{ReesSharp1978}
D. Rees and R. Y. Sharp,
\emph{On a theorem of B. Teissier on multiplicities of ideals in local
rings},
J. London Math. Soc. (2) \textbf{18} (1978), 449--463.


\bibitem{SarkarVerma}
P. Sarkar and J. K. Verma,
\emph{Local cohomology of multi-Rees algebras, joint reduction numbers and product of complete ideals},
Nagoya Math. J. \textbf{228} (2017), 1--20.

\bibitem{SwansonHuneke}
I. Swanson and C. Huneke,
\emph{Integral Closure of Ideals, Rings, and Modules},
London Math. Soc. Lecture Note Ser. \textbf{336},
Cambridge University Press, Cambridge, 2006.

\bibitem{Trung2001}
N. V. Trung,
\emph{Positivity of mixed multiplicities},
Math. Ann. \textbf{319} (2001), 33--63.

\bibitem{TrungVerma2007}
N. V. Trung and J. K. Verma,
\emph{Mixed multiplicities of ideals versus mixed volumes of polytopes},
Trans. Amer. Math. Soc. \textbf{359} (2007), no.~10, 4711--4727.

\bibitem{TrungVermaSurvey}
N. V. Trung and J. K. Verma,
\emph{Hilbert functions of multigraded algebras, mixed multiplicities
of ideals and their applications},
J. Commut. Algebra \textbf{2} (2010), no.~4, 515--565.

\bibitem{Viet2000}
D. Q. Viet,
\emph{Mixed multiplicities of arbitrary ideals in local rings},
Comm. Algebra \textbf{28} (2000), no.~8, 3803--3821.

\end{thebibliography}
\end{document}